\documentclass[11pt, a4paper]{article}

\usepackage{amsmath}
\usepackage{amsfonts}
\usepackage{amssymb}
\usepackage[francais,english]{babel}

\def\english{\selectlanguage{english}}

\providecommand\mathbb{\bf}
\newcommand\R{{\mathbb R}}

\newtheorem{thm}{Theorem}[section]
\newtheorem{lemma}{Lemma}[section]
\newtheorem{pro}{Proposition}[section]

\newtheorem{coro}{Corollary}[section]
\newtheorem{remark}{Remark}[section]

\newcounter{Remark}

\renewcommand\theRemark{\arabic{Remark}}

\newcounter{steps}
\newenvironment{proof}[1][]{%
\par\medbreak\setcounter{steps}{0}
{\noindent\bfseries Proof#1. }} {\hfill\fbox{\ }\medbreak}

\newcounter{substeps}[steps]

\newcommand{\intxv}[1]{
\int _{\R^2}\!\int _{\R ^2} \!\!#1 \;\mathrm{d}v\mathrm{d}x}

\newcommand{\intxvbis}[1]{
\int _{\R^4}\!\!#1 \;\mathrm{d}v\mathrm{d}x}

\newcommand{\intxvp}[1]{
\int _{\R^2}\!\int _{\R ^2} \!\!#1 \;\mathrm{d}\vp\mathrm{d}x}

\newcommand{\inttxv}[1]{
\int _{\R_+} \!\int _{\R^2}\!\int _{\R ^2} \!\!#1 \;\mathrm{d}v\mathrm{d}x\mathrm{d}t}

\newcommand{\intx}[1]{
\int _{\R ^2} \!#1 \;\mathrm{d}x}

\newcommand{\intv}[1]{
\int _{\R ^2} \!#1 \;\mathrm{d}v}

\newcommand{\intvp}[1]{
\int _{\R ^2} \!#1 \;\mathrm{d}v^{\prime}}

\newcommand{\vp}[0]{
v ^{\prime}}

\newcommand{\rp}[0]{
r ^{\prime}}

\newcommand{\eps}[0]{
\varepsilon}

\newcommand{\epsk}[0]{
{\varepsilon _k}}

\newcommand{\supe}[0]{
\sup _{\varepsilon >0}}

\newcommand{\fe}[0]{
f ^\varepsilon}

\newcommand{\re}[0]{
r ^\varepsilon}

\newcommand{\tre}[0]{
\tilde{r} ^\varepsilon}

\newcommand{\tfe}[0]{
\tilde{f} ^\varepsilon}

\newcommand{\fek}[0]{
f ^{\varepsilon _k}}

\newcommand{\limk}[0]{
\lim _{k \to +\infty}}

\newcommand{\fin}[0]{
f ^{\mathrm{in}}}

\newcommand{\Be}[0]{
B ^\varepsilon}

\newcommand{\Divx}[0]{
\mathrm{div}_x}

\newcommand{\Divv}[0]{
\mathrm{div}_v}

\newcommand{\Divy}[0]{
\mathrm{div}_y}

\newcommand{\Divxv}[0]{
\mathrm{div}_{x,v}}

\newcommand{\zm}[0]{
\{0,...,m-1\}}

\newcommand{\om}[0]{
\{1,...,m-1\}}

\newcommand{\ran}[0]{
\mathrm{Range\;}}

\newcommand{\bB}[0]{
{\bold B}}

\newcommand{\ltxv}[0]{
L^2(\R ^2 \times \R ^2)}

\newcommand{\ltm}[0]{
{L^2 _M}}

\newcommand{\ltF}[0]{
{L^2 _F}}

\newcommand{\liltF}[0]{
{L^\infty (\R_+; L^2 _F)}}

\newcommand{\ltmvx}[0]{
{L^2\left ( \frac{\mathrm{d}v\mathrm{d}x}{M(v)} \right )}}

\newcommand{\ltFvx}[0]{
{L^2\left ( \frac{\mathrm{d}v\mathrm{d}x}{F(x,v)} \right )}}

\newcommand{\oc}[0]{
\omega _c (x)}

\newcommand{\tc}[0]{
T_c (x)}

\newcommand{\ave}[1]{
\left \langle #1 \right \rangle }

\newcommand{\lty}[0]{
L^2( \R ^m)}

\newcommand{\inty}[1]{
\int _{\R ^m} \!\!\!\!#1 \;\mathrm{d}y}

\newcommand{\bin}[0]{
b^i \cdot \nabla _y}

\newcommand{\bzn}[0]{
b^0 \cdot \nabla _y}

\newcommand{\nxv}[0]{
\nabla _{x,v}}

\newcommand{\nxvp}[0]{
\nabla _{x,\vp}}

\newcommand{\nx}[0]{
\nabla _x}

\newcommand{\nv}[0]{
\nabla _v}

\newcommand{\nvp}[0]{
\nabla _{\vp}}

\newcommand{\tB}[0]{
\tilde{B}}

\newcommand{\tcalB}[0]{
\tilde{{\cal B}}}

\newcommand{\tQo}[0]{
\tilde{Q}^1}

\newcommand{\sumio}[0]{
\sum _{i =1 } ^{m-1}}

\newcommand{\sumiz}[0]{
\sum _{i =0 } ^{m-1}}

\newcommand{\sumjo}[0]{
\sum _{j =1 } ^{m-1}}

\newcommand{\sumjz}[0]{
\sum _{j =0 } ^{m-1}}

\newcommand{\an}[0]{
a \cdot \nabla _y }

\newcommand{\bjn}[0]{
b^j \cdot \nabla _y }

\newcommand{\D}[0]{
\mathrm{D}}

\newcommand{\dny}[0]{
\cdot \nabla _y}

\begin{document}
\english

\title{Impact of strong magnetic fields on collision mechanism for transport of charged particles}

\author{Mihai Bostan
\thanks{Laboratoire d'Analyse, Topologie, Probabilit\'es LATP, Centre de Math\'ematiques et Informatique CMI, UMR CNRS 7353, 39 rue Fr\'ed\'eric Joliot Curie, 13453 Marseille  Cedex 13
France. E-mail : {\tt bostan@cmi.univ-mrs.fr}}, Irene M.
Gamba\thanks{ICES \& Department of Mathematics, University of
Texas at Austin, Texas 78712 USA. E-mail: {\tt
gamba@math.utexas.edu}}
}

\date{ (\today)}

\maketitle

\begin{abstract}
One of the main applications in plasma physics concerns the energy production through thermo-nuclear fusion. The controlled fusion is achieved by magnetic confinement {\it i.e.,} the plasma is confined into a toroidal domain (tokamak) under the action of huge magnetic fields. Several models exist for describing the evolution of strongly magnetized plasmas, most of them by neglecting the collisions between particles. The subject matter of this paper is to investigate the effect of large magnetic fields with respect to a collision mechanism. We consider here linear collision Boltzmann operators and derive, by averaging with respect to the fast cyclotronic motion due to strong magnetic forces, their effective collision kernels.  
\end{abstract}



\section{Introduction}
\label{Intro}
\indent

The goal of this paper is to emphasize the role of the collision mechanism in the framework of magnetic confinement. We investigate both first and second order asymptotic approximations, when the magnetic field becomes large. We present completely explicit models and clearly indicate the form of the effective collision operator, the expression of the collision drift and check that the standard physical properties hold true.

Motivated by the magnetic fusion, many research programs in plasma physics 
concern the dynamics of charged particles under the action of strong magnetic fields $|\bB ^\eps| \to +\infty$ as $\eps \searrow 0$. Usually we appeal to kinetic models meaning that we replace the particles by their density (distribution) function $\fe = \fe (t,x,v)$ over the phase space, depending on time $t \in \R_+$, position $x\in \R^3$ and velocity $v \in \R^3$. The quantity $\fe (t,x,v) \;\mathrm{d}x\mathrm{d}v$ represents the particle number at time $t$, inside the infinitesimal volume $\mathrm{d}x\mathrm{d}v$ around $(x,v)$. The fluctuations of the density $\fe$ are due to the transport of the charged particles under the electro-magnetic field and also to the collisions between particles
\begin{equation}
\label{Equ1} \partial _t \fe + v \cdot \nabla _x \fe + \frac{q}{m} \left ( E(x) + v \wedge \bB ^\eps (x) \right ) \cdot \nabla _v \fe = Q(\fe),\;\;(t,x,v) \in \R_+ \times \R ^3 \times \R ^3.
\end{equation}
We add the initial condition
\begin{equation}
\label{Equ2}
\fe (0,x,v) = \fin (x,v),\;\;(x,v) \in \R ^3 \times \R ^3.
\end{equation}
Here $m$ is the particle mass and $q$ is the particle charge, $(E, \bB ^\eps)$ stands for the electro-magnetic field and $Q$ is the collision operator. We neglect the self-consistent electro-magnetic field; we assume that the magnetic field is stationary, divergence free and that the electric field derives from a given electric potential $E(x) = - \nabla _x \phi (x)$. We focus on the linear Boltzmann collision operator
\begin{equation}
\label{EquDomCol}
Q(f) (v) = \frac{1}{\tau} \int_{\R^3} s(v,\vp) \{ M(v) f(\vp) - M(\vp) f(v) \} \;\mathrm{d}\vp
\end{equation}
where $M$ is the Maxwellian equilibrium
\[
M(v) = \frac{1}{(2\pi \theta/m)^{3/2}} \exp \left (- \frac{m|v|^2}{2\theta}    \right )
\]
with temperature $\theta >0$, $\tau$ is the relaxation time and $s$ is the scattering cross section. We assume that
\begin{equation}
\label{Equ3} s(v,\vp) = s(\vp, v),\;\;v, \vp \in \R^3
\end{equation}
and
\begin{equation}
\label{Equ4} s_0 \leq s(v,\vp) \leq S_0,\;\;v, \vp \in \R^3
\end{equation}
for some $s_0, S_0 >0$. We analyze the asymptotic behaviour of (\ref{Equ1}) when the magnetic field becomes large
\begin{equation}
\label{Equ5}
\bB ^\eps  (x) = \frac{\bB (x)}{\eps},\;\;\bB (x)= B(x)b(x),\;\;\Divx (Bb) = 0,\;\;0 < \eps <<1
\end{equation}
for some scalar positive function $B(x)$ and some field of unitary vectors $b(x)$. 

These studies are motivated by the numerical simulations of strongly magnetized plasmas. Indeed, a direct resolution of \eqref{Equ1} using explicit numerical schemes requires a time discretization of order ${\cal O} (1/|\bB ^\eps |) = {\cal O}(\eps)$, and therefore a huge computation time, when $\eps$ becomes very small. The main idea will be to perform the asymptotic analysis of \eqref{Equ1} at the theoretical level {\it i.e.,} to replace \eqref{Equ1} by first/second order approximation models whose coefficients are stable when $\eps \searrow 0$. Eventually we solve numerically these asymptotic models. Notice that in this case the largest time step satisfying the stability condition does not depend on $\eps$, and therefore the numerical resolution remains efficient even for very large magnetic fields.

For simplicity we neglect the magnetic curvature {\it i.e.,} $\partial _x b = 0$. Assuming that $b = e_3$ and taking into account the divergence constraint in \eqref{Equ5} lead to $\bB ^\eps (x) = (0,0,\Be (x_1,x_2)) = (0, 0, B(x_1, x_2)/\eps)$. Therefore the equation \eqref{Equ1} should be understood in two dimensional setting
\begin{equation}
\label{Equ6} \partial _t \fe + v \cdot \nabla _x \fe + \frac{q}{m} \left ( E (x) + \frac{^\perp v}{\eps}  B(x) \right ) \cdot \nabla _v \fe = Q(\fe),\;\;t>0, (x,v) \in \R^2 \times \R^2
\end{equation}
where $^\perp v = (v_2, - v_1)$ and $Q$ is the two dimensional linear Boltzmann operator, associated to the Maxwellian $M(v) = m/(2\pi \theta)\; \exp ( - \frac{m|v|^2}{2\theta}), v \in \R^2$.

The dynamics of the particles is dominated by the transport operator  $\frac{1}{\eps} {\cal T}$ where  
\begin{equation}\label{TransT}
 {\cal T} = \frac{q B(x)}{m} {^\perp v } \cdot \nabla _v \, .
\end{equation}
When $\eps$ is small enough it is natural to introduce expansions like $\fe = f + \eps f^1 + \eps ^2 f ^2 + ...$ Plugging the previous Ansatz in \eqref{Equ6} gives at the lowest order the divergence constraint ${\cal T} f = 0$ and to the next order the evolution equation
\begin{equation}
\label{Equ7} \partial _t f + v \cdot \nabla _x f + \frac{q}{m} E (x) \cdot \nabla _v f + {\cal T} f ^1 = Q(f).
\end{equation}
We expect that the Boltzmann equation \eqref{Equ7} supplemented by the constraint ${\cal T} f = 0$ (and also by an appropriate initial condition) provides a well posed problem for the leading order term $f$. 

The approach is standard as in the Hilbert expansion methods of matching orders of $\eps$: the first order term $f^1$ enters \eqref{Equ7} as a Lagrange multiplier in such a way that at any time $t>0$ the constraint ${\cal T} f(t) = 0$ should be satisfied. Finding a closure for the dominant term $f$ reduces to eliminating the Lagrange multiplier provided that the constraint on $f$ holds true. Notice that the range and kernel of the operator ${\cal T}$ are orthogonal (in $L^2$) and therefore the closure for $f$ follows naturally by taking the orthogonal projection on $\ker {\cal T}$ of the Boltzmann equation \eqref{Equ7} at any fixed time $t>0$. One of the key point is to observe that taking the orthogonal projection on the kernel of ${\cal T}$ reduces to averaging along the flow associated to ${\cal T}$ cf. \cite{BosAsyAna}, \cite{BosTraSin}, \cite{BosGuidCent3D}. Observe that the motion associated to the vector field $(0,0,\oc   ^\perp v)$ is periodic
\[
X(s;x,v) = x,\;\;V(s;x,v) = R(-\oc  s)v
\]
where $\oc  = qB(x)/m$ is the rescaled cyclotronic frequency and for any $\alpha \in \R$ the notation $R(\alpha)$ stands for the rotation of angle $\alpha$
\begin{equation}\nonumber
R(\alpha) = \left (
\begin{array}{lll}
\cos \alpha  &  -\sin \alpha\\
\sin \alpha  &\;\;\;   \cos \alpha
\end{array}
\right ).
\end{equation}
We introduce the rescaled cyclotronic period $T_c (x) = 2\pi /\oc$. 
In this case, for any $L^2$ function $u$, the average simply writes
\begin{eqnarray}
\ave{u}(x,v) & = & \frac{1}{\tc} \int _0 ^ {\tc} u (X(s;x,v), V(s;x,v))\;\mathrm{d}s \nonumber \\
& = & \frac{1}{2\pi } \int _0 ^{2\pi} u (x, R(-\oc s)v)\;\mathrm{d}s.\nonumber
\end{eqnarray}
Applying the average operator to the Boltzmann equation \eqref{Equ7} yields the following closure for the dominant term $f$
\begin{equation}
\label{Equ8}
\ave{\partial _t f} + \ave{ v \cdot \nabla _x f + \frac{q}{m} E(x) \cdot \nabla _v f} = \ave{Q(f)}
\end{equation}
supplemented by the constraint ${\cal T} f(t) = 0$ at any time $t >0$. Notice that $\{\psi _1 = x_1, \psi _2 = x_2, \psi _3 = |v|\}$ form a complete family of prime integrals for ${\cal T}$ and thus the constraint ${\cal T}f = 0$ reads $f(t,x,v) = g(t, x, r = |v|)$ for some function $g = g(t,x,r)$. Straightforward computations show that 
\[
\ave{\partial _t f} = \partial _t \ave{f} = \partial _t f,\;\;\ave{v \cdot \nabla _x f} = \ave{v \cdot \nabla _x g (t,x,|v|)} = 0,\;\;\ave{E \cdot \nabla _v f} = \ave{E \cdot \frac{v}{|v|} \partial _r g } = 0.
\]
Moreover, we will see that the linear collision operator commutes with the average operator {\it i.e.,} $\ave{Q(f)} = Q(\ave{f}) = Q(f)$ and finally the dominant term satisfies the space homogeneous Boltzmann equation \cite{Cer88}, \cite{CerIllPul94}, \cite{MarRinSch90}
\[
\partial _t f = Q(f).
\]
In particular there is no current $\intv{vf(t,x,v)} = \intv{v g(t,x,|v|)} = 0$ along the orthogonal directions to the magnetic lines, as expected by the magnetic confinement context. But orthogonal drifts may appear at the next order. It is well known that the particles move slowly along the orthogonal directions and we distinguish the electric cross field drift $v _\wedge ^\eps = \frac{^\perp E}{\Be}$, the magnetic gradient drift $v_{\mathrm{GD}} ^\eps = -\frac{|v|^2}{2\omega _c ^\eps } \frac{^\perp \nabla _x \Be}{\Be}, \omega _c ^\eps = \frac{\omega _c}{\eps}$, the magnetic curvature drift (which is neglected here since $\partial _x b = 0$) cf. \cite{Rax}. We intend to investigate the effect of strong magnetic fields through the collision mechanism. We will also prove that the collisions account for a drift motion which, to the best of our knowledge, has not been reported yet. In the case of a constant scattering cross section $s$ we find the formula
\[
v _{\mathrm{QD}} ^\eps = - \frac{s}{\tau} \frac{^\perp v}{\omega _c ^\eps}
\]
(see \eqref{EquDriftNonConstantS} for non constant scattering cross section). Naturally, in order to observe the previous drifts (of order of $\eps$) we need to take into account the first order correction term in the expansion $\fe = f + \eps f^1 + \eps ^2 f^2 + ...$ eventually to appeal to the next order evolution equation
\begin{equation}
\label{Equ9}
\partial _t f^1 + v \cdot \nabla _x f^1 + \frac{q}{m} E \cdot \nabla _v f^1 + {\cal T} f^2 = Q(f^1).
\end{equation}
As before we are looking for a closure with respect to $f + \eps f^1$, by getting rid of the second order correction $f^2$. This will be achieved using again the averaging technique and combining properly the time evolutions of $f, f^1$ in \eqref{Equ7}, \eqref{Equ9}. Doing that we will get another Boltzmann like equation, whose collision operator $Q + \eps Q^1$ is a first order perturbation of the standard linear Boltzmann operator $Q$. The interesting point is that the new collision operator is non local in space anymore. For example, when the scattering cross section is constant we obtain
\[
\eps Q^1 (f) (x,v) = \frac{s}{\tau} M(v) \left \{m \;\frac{v \cdot v_\wedge ^\eps }{\theta} \;\rho _f (x) - \frac{v}{\omega _c ^\eps} \cdot  \;^\perp \nabla _x \rho _f + \Divx \left (  \frac{^\perp j_f}{\omega _c ^\eps } \right )   \right \}
\]
where $\rho _f (x) = \intvp{f(x,\vp)}, j _f (x)  = \intvp{\vp f(x,\vp)}$.
Surely the construction of the new collision operator should take into account the classical physical properties, as mass conservation and entropy inequality . We will see that $\eps Q^1$ (and therefore $Q + \eps Q^1$ as well) satisfies  global mass conservation and entropy inequality.  

We point out that a linearized and gyroaveraged collision operator has been written in \cite{XuRos91}, but the implementation of this operator seems very hard. We refer to \cite{Bri04} for a general guiding-center bilinear Fokker-Planck collision operator. Another difficulty lies in the relaxation of the distribution function towards a local Maxwellian equilibrium. Most of the available model operators, in particular those which are linearized near a Maxwellian, are missing this property. Very recently a set of model collision operators has been obtained in \cite{GarDifSarGra09}, based on entropy variational principles. 

The analysis of the Vlasov or Vlasov-Poisson equations with large external magnetic field have been carried out in \cite{FreSon98}, \cite{GolSai99}. The numerical approximation of the gyrokinetic models has been performed in \cite{GraBruBer06} using
semi-Lagrangian schemes, see also \cite{Gar97}, \cite{Gar98}. 
For recent results on the Boltzmann equation we refer to \cite{BosGamGou10}, \cite{BolVlaPoi1D}. High electric field limits of the Boltzmann-Poisson system have been addressed in \cite{Pou92}, \cite{CerGamLev97}, \cite{CerGamLev01}, \cite{BenGamKla04}.

The nonlinear  gyrokinetic theory of the Vlasov-Maxwell equations can be carried out by appealing to Lagrangian and Hamiltonian methods \cite{BriHam07}, \cite{Little79}, \cite{Little81}. It is also possible to follow the general method of multiple time scale or averaging perturbation developped in \cite{BogMit61}. For a unified treatment of the main physical ideas and theoretical methods that have emerged on magnetic plasma confinement we refer to \cite{HazMei03}.

The outline of the paper is the following. In Section \ref{MainResults} we introduce our models and present the main results. Section \ref{AveTec} is devoted to the definition and main properties of the average operator. In particular we investigate its commutation properties with respect to first order differential operators and the collision operator. The first and second order asymptotic models for the linear Boltzmann equation with large magnetic field are detailed in Section \ref{LimBolEqu}.

\section{Presentation of the models and main results}
\label{MainResults}
\indent

As observed formally before, the first order approximation of \eqref{Equ6} leads to the space homogeneous Boltzmann equation. We obtain rigorous weak and strong convergence results, at least for electric potential satisfying
\begin{equation}
\label{Equ49} E(x) = - \nabla _x \phi (x),\;\;\intx{\exp \left ( - q\frac{\phi (x)}{\theta} \right ) } <+\infty.
\end{equation}
The equilibrium $F(x,v) = M(v) \exp \left ( - q\phi (x)/\theta \right ) $ belongs to $\ltF = \ltFvx$ and by standard computations we obtain the uniform bound
\[
\intxv{\frac{(\fe (t,x,v))^2}{F(x,v)}} \leq \intxv{\frac{(\fin (x,v))^2}{F(x,v)}},\;\;\eps >0
\]
which allows us to get weak compactness in $\liltF$. For scattering cross section of the form $s(v,\vp) = \sigma ( |v - \vp|)$ one has
\begin{thm}
\label{WeakConv} Assume that the scattering cross section has the form $s(v,\vp) = \sigma ( |v - \vp|) \in [s_0, S_0], v, \vp \in \R^2$ and that the electrostatic potential verifies \eqref{Equ49}. For any $\eps >0$ we denote by $\fe$ the weak solution of \eqref{Equ6}, \eqref{Equ2} associated to a initial condition satisfying $\fin \in \ltF$. Therefore $(\fe)_{\eps >0}$ converges weakly $\star$ in $\liltF$ as $\eps \searrow 0$ towards the solution of the problem
\begin{equation}
\label{Equ55} \partial _t f = Q(f),\;\;(t,x,v) \in \R_+ ^\star \times \R^2 \times \R^2,\;\;f(0,x,v) = \ave{\fin}(x,v),\;\;(x,v) \in \R^2 \times \R^2.
\end{equation}
\end{thm}
Strong convergence results hold true as well. The main idea is to introduce a smooth first order corrector, under appropriate smoothness hypotheses on the initial condition.
\begin{thm}
\label{ConvStrong} Assume that the scattering cross section satisfies $s(v,\vp) = \sigma (|v- \vp|) \in [s_0, S_0], v, \vp \in \R ^2$ for some function $\sigma \in W^{2,\infty}$, that the electro-magnetic field $(E(x) = - \nabla _x \phi, B(x))$ belongs to $W^{1,\infty} (\R^2)$ such that
\[
\intx{\exp \left ( - \frac{q\phi (x)}{\theta} \right ) } < +\infty,\;\;\inf _{x \in \R^2} B(x) >0
\]
and that the initial condition verifies
\[
{\cal T} \fin = 0,\;\;( 1 + |v|^2 ) \;|\nabla ^2 _{x,v} \fin |\in \ltF.
\]
Then for any $T>0$ there is a constant $C_T$ such that 
\[
\|\fe (t) - f(t) \|_\ltF \leq C_T \eps,\;\;t \in [0,T],\;\;\eps >0.
\]
\end{thm}
The leading order term in the expansion of $\fe$ satisfies, at any fixed space point, the homogeneous Boltzmann equation. There is no phase space transport, the collision operator $Q$ is local in space and there is no current, since $f(t,x,\cdot)$ depends only on $|v|$. A much more interesting model will be the second order approximation, since we expect to retrieve the usual drifts along the orthogonal directions with respect to the magnetic lines, but also a collision drift. Certainly, the second order approximation will appear as a first order perturbation of the space homogeneous Boltzmann equation, involving a new transport operator but also a new collision operator $Q + \eps Q^1$, which is non local in space anymore. The full expression of the operator $Q^1$ is rather complicated, see Corollary \ref{Q1}, but it simplifies a lot when the scattering cross section is constant
\[
Q^1(f) = \frac{s}{\tau} M(v) \left [\Divx \left ( \frac{^\perp j_f (x)}{\oc} \right ) + \frac{\rho _F (x)}{\oc} \;^\perp v \cdot \nx \left ( \frac{\rho _f (x)}{\rho _F (x)} \right )    \right ]
\]
where $\rho _f = \intvp{f}$, $\rho _F = \intvp{F}$, $j_f = \intvp{\vp f}$.
For the sake of the presentation we formulate the second order convergence result in the context of constant scattering cross section but we point out that similar results hold true in the general case.
\begin{thm}
\label{ThSecOrdApp}
Assume that the electro-magnetic field $(E(x) = - \nabla _x \phi (x), \Be (x) = B(x)/\eps)$ is smooth such that 
\[
\intx{\exp \left ( - \frac{q\phi (x)}{\theta} \right ) } <+\infty,\;\;\inf _{x \in \R^2} B(x) >0.
\]
We denote by $\fe$ the weak solution of \eqref{Equ6} with the initial condition $\fe (0) = \fe _{\mathrm{in}}$ and by $\tfe$ the weak solution of
\begin{eqnarray}
& & \label{Equ91} \partial _t \tfe + \left ( \frac{^\perp E}{\Be} - \frac{|v|^2}{2\omega _c ^\eps } \;\frac{^\perp \nabla _x \Be}{\Be}\right ) \cdot \nx \tfe + \left ( \frac{^\perp E}{2\Be} \cdot \frac{\nabla _x \Be}{\Be} \right ) v \cdot \nv \tfe \nonumber \\
& & = \frac{s}{\tau} M(v) \left ( \rho _{\tfe} - \frac{\tfe}{M(v)} - \frac{\rho _F}{\omega _c ^\eps} \;v \cdot \;^\perp \nabla _x \left ( \frac{\rho _{\tfe}}{\rho _F} \right ) + \Divx \left ( \frac{^\perp j _{\tfe}}{\omega _c ^\eps} \right ) \right )  
\end{eqnarray}
with the initial condition
\begin{equation}
\label{Equ92} \tfe (0) = \fin + \frac{^\perp v}{\omega _c ^\eps} \cdot \nx \fin - \frac{^\perp E}{\Be} \cdot \nabla _v \fin,\;\;{\cal T} \fin = 0
\end{equation}
where $\omega _c ^\eps = \frac{q\Be}{m} = \frac{\omega _c}{\eps}$ is the cyclotronic frequency. If the family of the initial conditions $(\fe _{\mathrm{in}})_{\eps > 0}$ satisfies
\[
\fe _{\mathrm{in}} = \fin + \frac{^\perp v }{\omega _c ^\eps} \cdot \nx \fin - \frac{^\perp E}{\Be} \cdot \nabla _v \fin + {\cal O} (\eps ^2)
\]
then we have $\fe = \tfe + {\cal O} (\eps ^2)$.
\end{thm}
\begin{remark}
Notice that the coefficients in \eqref{Equ91} still depend on $\eps$ but they are stable when $\eps \searrow 0$. In particular the numerical approximation can be achieved uniformly with respect to $\eps>0$, using a time step not depending on $\eps$. 
\end{remark}

\section{Averaging along a characteristic flow}
\label{AveTec} 
\indent

This section is devoted to the average technique. It consists in averaging along the characteristic flow associated to a divergence free vector field. When applied to the gyrokinetic theory, this method is known as the gyroaverage technique. It allows us to derive limit models by neglecting the fluctuations associated to fast motions, as the cyclotron rotation of particles around the magnetic lines. From the mathematical point of view, it reduces to averaging along the characteristic flow associated to the dominant advection field (here the magnetic force in the framework of magnetic confinement) and it is also equivalent to projecting onto the subspace of functions which are left invariant by the same flow. These tools are well known and have been used when establishing the guiding center approximation in three dimensions for general magnetic shapes \cite{BosGuidCent3D}, the finite Larmor radius regime \cite{BosAsyAna} and, more generally, when studying transport equations with disparate advection fields \cite{BosTraSin}. We point out that the averaging technique applies to many other cases : models which take into account the mass ratio between electrons and ions subject to strong magnetic fields \cite{BosNeg09} and models with fast oscillating magnetic fields \cite{Bos12}. In the present work we generalize the averaging method to collisional models. Motivated by the study of the Boltzmann equation, we revisit some results concerning the average of first order differential operators which will combine with other new results in order to study the average of collision operators.

\subsection{Gyroverage operator}
\label{AveOpe}
\indent

We recall the definition and the properties of the average operator corresponding to the transport operator ${\cal T}$ given in \eqref{TransT},  now in divergence form, whose definition  in the $\ltxv$ setting is given by  
$$
{\cal T}u = \mathrm{div}_v \left ( \omega _c (x) \;u(x,v) \; ^\perp v\right ),\;\;\omega _c (x) = \frac{qB(x)}{m}
$$
for any function $u$ in the domain
$$
\D({\cal T}) = \{ u (x,v) \in \ltxv{}\;:\; \mathrm{div}_v \left ( \omega _c (x) \;u \;^\perp v \right ) \in \ltxv{}\}.
$$
We denote by $\|\cdot \|$ the standard norm of $\ltxv{}$ and by $(X,V)(s;x,v)$ the characteristics associated to $\omega _c (x) \;^\perp v  \cdot \nabla _v$
\begin{equation}
\label{Equ11}
\frac{\mathrm{d}X}{\mathrm{d}s} = 0,\;\;\frac{\mathrm{d}V}{\mathrm{d}s} = \omega _c (X(s)) \;^\perp V(s),\;\;(X,V)(0) = (x,v).
\end{equation}
By direct computations we deduce that 
\begin{equation}
\label{EquFlow}
X(s;x,v) = x, \;\;V(s;x,v) = R(-\oc s)v.
\end{equation}
A complete family of invariants for the flow \eqref{EquFlow} is given by $\{x_1, x_2, |v| = ((v_1)^2 + (v_2)^2) ^{1/2}\}$ and therefore the kernel of ${\cal T}$ is given by those functions in $\ltxv{}$ depending only on $x_1, x_2$ and $|v|$. Notice that all trajectories are $\tc = 2\pi/\oc $ periodic and we introduce the average operator cf. \cite{BosTraSin}
\begin{eqnarray}
\ave{u}(x,v) & = & \frac{1}{\tc} \int _0 ^ {\tc} u (X(s;x,v), V(s;x,v))\;\mathrm{d}s \nonumber \\
& = & \frac{1}{2\pi } \int _0 ^{2\pi} u (x, R(-\alpha)v) \;\mathrm{d}\alpha \ =\  \frac{1}{2\pi } \int _0 ^{2\pi} u (x, R(\alpha)v) \;\mathrm{d}\alpha \nonumber
\end{eqnarray}
for any function $u \in \ltxv{}$. The following two results are justified in \cite{BosGuidCent3D}, Propositions 2.1, 2.2. The first one states that averaging reduces to orthogonal projection onto the kernel of ${\cal T}$. The second one concerns the invertibility of ${\cal T}$ on the subspace of zero average functions and establishes a Poincar\'e inequality.
\begin{pro}
\label{AverageProp} The average operator is linear continuous. Moreover it coincides with the orthogonal projection on the kernel of ${\cal T}$ {\it i.e.,}
$$
\ave{u} \in \ker {\cal T}\ \;\mathrm{and}\ \;\intxv{(u - \ave{u}) \varphi } = 0,\;\;\forall \; \varphi \in \ker {\cal T}.
$$
\end{pro}
\begin{remark} Notice that we have the orthogonal decomposition of $\ltxv{}$ into invariant functions along the characteristics \eqref{EquFlow} and zero average functions
\[
u = \ave{u} + (u - \ave{u}),\;\;\intxv{(u - \ave{u})\ave{u}} = 0.
\]
Taking into account that ${\cal T} ^\star = - {\cal T}$, the equality $\ave{\cdot} = \mathrm{Proj}_{\ker {\cal T}}$ implies 
\[
\ker \ave{\cdot} = (\ker {\cal T})^\perp = (\ker {\cal T}^\star )^\perp = \overline{\ran {\cal T}}.
\]
In particular $\ran {\cal T} \subset \ker \ave{\cdot}$.
\end{remark}
\medskip
The next point is very important as it relates to  the solvability of ${\cal T}u = v$. We prove that the range of ${\cal T}$ is closed and  derive the corresponding Poincar\'e type inequality, provided the magnetic field remains away from $0$.
\begin{pro}
\label{TransportProp}
We assume that $\inf _{x \in \R ^2} B (x)  >0$. Then ${\cal T}$ restricted to $\ker \ave{\cdot}$ is one to one map onto $\ker \ave{\cdot}$. Its inverse belongs to ${\cal L }(\ker \ave{\cdot}, \ker \ave{\cdot} )$ and we have the Poincar\'e inequality
\[
\|u \| \leq \frac{2 \pi }{|\omega _0 |} \|{\cal T} u \|,\;\;\omega _0 = \frac{q}{m} \inf _{x \in \R ^2} B(x) \neq 0
\]
for any $u \in \D({\cal T}) \cap \ker \ave{\cdot}$.
\end{pro}
Another useful result is given by
\begin{pro}
\label{IntByPart}
Assume that $\varphi $ is a $C^1(\R^2 \times \R^2)$ function such that $\varphi$ and ${\cal T} \varphi $ are bounded on $\R^2 \times \R^2$. Then, for any function $u \in \D ({\cal T}) $ we have the integration by parts formula
\[
\ave{\varphi \;{\cal T}u} + \ave{u \;{\cal T} \varphi } = 0.
\]
\end{pro}
\begin{proof}
We have $\varphi u \in \D({\cal T})$ and ${\cal T} (\varphi u) = \varphi \;{\cal T} u + u \;{\cal T} \varphi $. Taking the average yields
\[
\ave{\varphi  \;{\cal T} u } + \ave{u \;{\cal T} \varphi } = \ave{{\cal T}(\varphi u )} = 0
\]
since $\ran {\cal T} \subset \ker \ave{\cdot}$.
\end{proof}

\subsection{Averaging and first order differential operators}
\label{AveDiff}
\indent

In the sequel we need to analyze the commutation properties of averaging and transport operators. For the sake of the presentation we perform our analysis in the general framework of characteristic flows associated to divergence free vector fields. We denote by $b^0$ a given smooth, divergence free vector field $b^0 : \R^m \to \R^m$ whose characteristic flow $Y = Y(s;y)$ is well defined
\begin{equation}
 \label{H00}
 \left\lbrace \begin{array}{l l}
 \frac{d}{ds}Y &= b^0(Y(s;y)),\;\;(s,y) \in \R \times \R 
^m \\
Y(0;y) &= y,\;\; y \in \R ^m.
\end{array}
 \right.
\end{equation}
We intend to apply the following general results to the gyroaverage context and therefore, without loss of generality, we will assume that the characteristic flow $Y=Y(s;y)$ is periodic {\it i.e.,} for any $y \in \R^m$ there is $T_c(y) >0$ such that 
\begin{equation}
\label{H0} Y(T_c(y);y) = y. 
\end{equation}
We may assume that $T_c(y)$ is the smallest positive period, excepted at those points $y \in \R^m$ where $b^0 (y) = 0$ when $s \to Y(s;y)$ is constant. The non periodic case is beyond the scope of this work. Nevertheless notice that it can be handled as well, the average operator associated to a non periodic flow being defined through ergodic mean (see \cite{BosTraSin} for details)
\[
\ave{u}(y) = \lim _{T \to +\infty} \frac{1}{T} \int _0 ^T u (Y(s;y))\;\mathrm{d}s,\;\;y \in \R^m.
\]
For the sake of clarity, the results formulated in the general framework are stated by using latin upper case $T, A, B, C, ...$ whereas their counterpart results, referring to the gyroaverage framework, are stated by using calligraphy upper case ${\cal T, A, B, C}, ...$.
\begin{pro}
The map  $y \to T_c(y)$ is constant along the flow $Y$. 
\end{pro}
\begin{proof}
Indeed, for any $y \in \R^m$ such that $b^0 (y) \neq 0$ let us consider $y_h = Y(h;y)$, with $h \in \R$. Clearly $b^0(y_h) \neq 0$ (otherwise $Y(\cdot;y) = y$ and $b^0 (y) = b^0 (y_h) = 0$) and 
\[
Y(T_c(y);y_h) = Y(T_c(y) + h;y) = Y(h;Y(T_c(y);y)) = Y(h;y) = y_h\, ,
\]
which implies $T_c(y_h) \leq T_c(y)$. Replacing now $y$ by $y_{-h}$ (which satisfies $b^0 (y_{-h}) \neq 0 $) we deduce that $T_c(y) \leq T_c(y_{-h})$ for any $h \in \R$ and therefore, after changing $h$ by $-h$ in the last inequality we obtain $T_c(y_h) \leq T_c(y) \leq T_c(y_h)$, saying that $T_c(y)$ is constant along the flow $Y$.
\end{proof}

Next, we denote by $T$ the linear operator defined by $T u = \Divy (u(y) b^0(y) )$ for any $u$ in the domain
\begin{equation*}
 \mathrm{D} (T) = \{ u \in \lty{}\;:\; \Divy (b^0 (y) u(y))  \in \lty{} \}\, ,
\end{equation*}
and thus the kernel of $T$ is given by constant functions along the characteristic flow \eqref{H00}
\begin{equation}\label{H-T}
\ker T = \{ u \in \lty{}\;:\; u (Y(s;y)) = u(y),\;s\in
\R,\;\mathrm{a.e.}\; y \in \R^m\}.
\end{equation}
\medskip
For further computations it is very useful to pick appropriate coordinates. Since we are interested in averaging along a characteristic flow (and therefore since the functions left invariant along this flow play a crucial role) we take as new coordinates a complete family of invariants $\{\psi _1, ..., \psi _{m-1}\}$ for the field $b^0$, together with a parametrization $\psi _0$ along its flow. To any invariant $\psi _i, i \in \{1,...,m-1\}$ we associate a derivation $b^i \cdot \nabla _y $ acting, with respect to the new coordinates, like the partial derivative $\partial _{\psi _i}$. Assume that the vector field $b^0$ has a complete family of smooth  independent prime integrals denoted $\psi _1 = \psi _1 (y), ..., \psi _{m-1} = \psi _{m-1} (y)$ {\it i.e.,}
\begin{equation}
\label{H1} b^0 \cdot \nabla _y  \psi _i = 0,\;\;i \in \{1,...,m-1\}, \;\;\mathrm{rank} \;\partial _y   \; ^t (\psi _1, ..., \psi _{m-1}) = m-1,\;\;y \in \R^m\,.
\end{equation}
Moreover, assume that there is a function $\psi _0 = \psi _0 (y)$, whose gradient $\nabla _y \psi _0$ is well defined a.e. $y \in \R^m$, such that 
\begin{equation}
\label{H2} b^0 \cdot \nabla _y \psi _0 = I(y) \neq 0,\;\;\mathrm{a.e.}\; y \in \R^m
\end{equation} 
for some function $I \in \ker T$. Notice that $\psi _0$, satisfying \eqref{H2}, cannot be continuous everywhere. In general, $\psi _0$ has a discontinuity point,  say $y_0$, on each characteristic and the jump $[\psi _0]$ around the discontinuity is given by
\begin{equation}\label{H3}
\lim _{s \nearrow T_c(y_0)} \psi _0 (Y(s;y_0)) - \lim _{s \searrow 0 } \psi _0 (Y(s;y_0)) = \int _0 ^{T_c (y_0)} (b^0 \dny \psi _0 )(Y(s;y_0)) \;\mathrm{d}s = T_c(y_0) I(y_0).
\end{equation}
That means $\psi _0 (y)$ is an angular coordinate. Since $T_c(\cdot), I(\cdot) \in \ker ( b^0 \dny)$ it is possible to normalize the jump. Indeed, taking $\tilde{\psi}_0 = \psi _0 /[\psi _0]$ leads to
\[
b^0 \dny \tilde{\psi}_0 = \frac{b^0 \dny \psi _0}{[\psi _0]} = \frac{I(y)}{T_c(y) I(y)} = \frac{1}{T_c(y)} \in \ker (b^0 \dny ),\;\;[\tilde{\psi}_0] = 1.
\]
Therefore we assume, without loss of generality, that the angular coordinate $\psi _0$ has the same jump on every characteristic. The following result is standard. For the sake of the completeness we detail the proof in Appendix \ref{A}.
\begin{pro}
\label{prime-integrals}
Assume that the vector field $b^0$ has a complete family of  smooth independent prime integrals \eqref{H1}. Moreover we assume that there is an angular coordinate \eqref{H2} with constant jump $S = T_c I $ on every characteristic. Then, the following two properties are satisfied
\begin{enumerate}
\item
The map $y \to (\psi _0 (y), \psi _1(y),...,\psi _{m-1} (y))$ is a bijection between $\R^m$ and $[0,S) \times D$, 
where  $D := \{(\psi _1(y),...,\psi _{m-1} (y))\;:\; y \in \R^m\}$.
\item 
For any $i \in \{1, ..., m-1\}$ there is a unique vector field $b^i$ such that 
\[
\bin \psi _j = \delta ^i _j,\;\;j \in \{0,1,...,m-1\}.
\]
\end{enumerate}
\end{pro}
\begin{remark}
\label{PerCont} 
Any function $u = u(y)$ writes $u(y) = w(\psi _0 (y), \psi _1(y),...,\psi _{m-1} (y))$, with $w$ defined on $[0,S) \times D$. Notice that if $u$ is continuous then $\lim _{r \nearrow S} w(r,\psi _1,...,\psi _{m-1}) = w(0,\psi _1,...,\psi _{m-1})$. Indeed, for any $(\psi _1,...,\psi _{m-1}) \in D$ let us denote by $y_0$ the discontinuity point of $\psi _0$ along the characteristic associated to the invariants $(\psi _1,...,\psi _{m-1})$. We have
\begin{eqnarray}
w(0, \psi _1,...,\psi _{m-1}) &=& u(y_0)  =   \lim _{s \nearrow T_c(y_0)} u(Y(s;y_0)) \nonumber \\
& = & \lim _{s \nearrow  T_c (y_0)} w(\psi _0(Y(s;y_0)),\psi _1 (y_0),...,\psi _{m-1}(y_0)) \nonumber \\
& = & \lim _{s \nearrow T_c (y_0)} w(\psi _0 (y_0) + sI, \psi _1,...,\psi _{m-1}) \nonumber \\
& = & \lim _{r \nearrow S } w(r, \psi _1,...,\psi _{m-1}). \nonumber 
\end{eqnarray}
If $u$ is continuous, the function $w$ above extends by $S$- periodicity with respect to $\psi _0$ to a continuous function, still denoted $w$, defined on $\R \times D$.
\end{remark}
\begin{remark}
\label{PartCase} In our case \eqref{Equ6} the divergence free vector field to be considered is $b^0 :\R^4 \to \R^4$, $b^0 (x,v) = (0,0,\oc ^\perp v )$. A complete family of prime integrals for this vector field is
\[
\psi _1 = x_1,\;\;\psi _2 = x_2,\;\;\psi _3 = |v|
\]
and we pick as angular coordinate $\psi _0 (x,v) = - \alpha (v)$ where $\alpha (v)$ is the angle of $v \in \R^2 \setminus \{(0,0)\}$ {\it i.e.,}
\[
v_1 = |v| \cos \alpha (v),\;\;v_2 = |v| \sin \alpha (v),\;\;\alpha (v) \in [0,2\pi).
\]
The angular coordinate  $\psi _0 (x,v)$ is discontinuous on $\R_+ ^\star \times \{0\}$. But its gradient, which is well defined for any $v \in \R ^2 \setminus (\R_+ \times \{0\})$, can be extended by continuity on $\R^2 \setminus \{(0,0)\}$
\[
\nabla _v \alpha = - \frac{^\perp v}{|v|^2},\;\; \nxv \psi _0 = \left ( 0,0,\frac{^\perp v}{|v|^2} \right )
\]
and we have $I = \oc \;^\perp v \cdot \nabla _v  \psi _0 = \oc \in \ker {\cal T}$. On each characteristic $\psi _0$ has a jump of $S = \tc \oc = 2\pi$.

In addition, the vector fields $({b}^i)_{i \in \{1,2,3\}}$ associated to the invariants $\psi _1 = x_1, \psi _2 = x_2, \psi _3 = |v|$ and the angular coordinate $\psi _0 = - \alpha (v)$ are given by 
\[
{b}^1 = (1,0,0,0),\;\;{ b}^2 = (0,1,0,0),\;\;{ b}^3 = \left ( 0,0,\frac{v}{|v|} \right ).
\]
\end{remark}

\medskip

The vector fields $(b^i)_{i \in \{1,...,m-1\}}$ will play a major role in our analysis cf. \cite{BosGuidCent3D}. We shall establish several general results which mainly concern the commutation properties between the average and divergence operators, the inversion of $T = b^0 \cdot \nabla _y$ on zero average functions and the construction of some auxiliary vector fields we will need for our further analysis. In particular we will show that after averaging, a transport operator reduces to another transport operator and we explicitely compute its advection field components. The following eight propositions explain in detail how to achieve these goals. Observe that the derivations along the vector fields $(b^i)_{i \in \{1,...,m-1\}}$ coincide with the partial derivatives with respect to the invariants $(\psi_i)_{i \in \{1,...,m-1\}}$, see Appendix \ref{A} for details.
\begin{pro}
\label{Actions} Consider a function $u = u(y)$ and let us denote by $ w = w(\psi _0,...,\psi _{m-1}) : \R \times D \to \R$ the $S$ periodic function with respect to $\psi _0$ such that $u(y) = w(\psi _0 (y),...,\psi _{m-1} (y))$. If $u \in C^1 (\R^m)$, then $w \in C^1 (\R \times D)$ and satisfies the identities
\[
\bzn u (y) = I(y) \;\partial _{\psi _0} w(\psi _0(y),...,\psi _{m-1} (y)),\;\; \bin u (y) = \partial _{\psi _i} w(\psi _0(y),...,\psi _{m-1} (y))
\]
for any $i \in \{1,...,m-1\}$.
\end{pro}
In the sequel, for any two smooth vector fields $b(y)$ and $c(y)$, the notation $[b,c]$ stands for the vector field $(b \cdot \nabla _y )c - (c \cdot \nabla _y )b$, which is the vector field of the transport operator $b\cdot \nabla _y (c \cdot \nabla _y ) - c \cdot \nabla _y (b \cdot \nabla _y )$
\[
b\cdot \nabla _y (c \cdot \nabla _y ) - c \cdot \nabla _y (b \cdot \nabla _y ) = [b,c] \cdot \nabla _y.
\]
The commutators between the fields $(b^i)_{i \in \{0,...,m-1\}}$ come easily by Schwarz's theorem on mixed partial derivatives \cite{Rud76}.
\begin{pro}
\label{ComFields} With the previous notations we have
\[
[b^i, b^0] = \frac{\bin I}{I} b^0,\;\;i \in \{0,...,m-1\}
\]
and $[b^i,b^j] = 0, i,j \in \{1,...,m-1\}$, with $I = \bzn \psi _0$.
\end{pro}
\begin{proof}
The first statement is obvious when $i = 0$. For any smooth function $u = u(y)$ and $i \in \{1,...,m-1\}$ we have, after the change of coordinate $u(y) = w(\psi (y)), \psi = (\psi _0,...,\psi _{m-1})$
\begin{eqnarray}
[b^i,b^0]\dny u & = & (\bin )(\bzn u) - (\bzn) (\bin u) \nonumber \\
& = & (\bin ) \left ( I \frac{\partial w}{\partial {\psi _0}} \circ \psi \right ) - (\bzn ) \left (  \frac{\partial w}{\partial {\psi _i}} \circ \psi \right ) \nonumber \\
& = & \frac{\bin I}{I}\; \bzn u.\nonumber
\end{eqnarray}
Similarly, for any $i,j \in \{1,...,m-1\}$ we have
\begin{eqnarray}
[b^i,b^j] \cdot \nabla _y u  & = & (\bin) (\bjn u ) - (\bjn ) (\bin u) \nonumber \\
& = &  \left ( \frac{\partial }{\partial {\psi _i}} \left ( \frac{\partial w}{\partial {\psi _j}} \right ) \right ) \circ \psi - \left ( \frac{\partial }{\partial {\psi _j}} \left ( \frac{\partial w}{\partial {\psi _i}} \right ) \right ) \circ \psi = 0. \nonumber 
\end{eqnarray}
\end{proof}
\begin{coro}
\label{C} The operators $(\bin)_{i \in \{0,...,m-1\}}$ leave invariant $\ker \bzn$.
\end{coro}
\begin{proof}
By Proposition \ref{ComFields} we have for any $u \in \ker (\bzn) \cap \D ( \bin) $
\[
(\bin) (\bzn u) - (\bzn) (\bin u) = \frac{\bin I}{I} (\bzn u)
\]
and therefore $(\bzn) (\bin u ) = 0$.
\end{proof}
We establish now some properties regarding the divergences of the fields $(b^i)_{i \in \{1,...,m-1\}}$.
\begin{pro}
\label{FieldsDiv} The matrix $(\bin (\Divy b^j))_{(i,j) \in \{0,...,m-1\}}$ is symmetric. In particular, for any $i \in \{1,...,m-1\}$ the divergence of $b^i$ is constant along the flow of $b^0$.
\end{pro}
\begin{proof}
By Proposition \ref{ComFields} we know that
\[
[b^i, b^j] = \lambda ^{ij} b^0,\;\;\bzn \lambda ^{ij} = 0,\;\;i,j \in \zm
\]
since $I$ and $\bin I$ are constant along the flow of $b^0$ cf. Corollary \ref{C}. Therefore we have
\[
0 = \Divy (\lambda ^{ij} b^0) = \Divy [b^i,b^j] = (\bin) (\Divy b^j ) - (\bjn) (\Divy b^i ).
\]
In particular, taking $j = 0$ and $i \in \om$ one gets $\bzn (\Divy b^i) = 0$, since $b^0$ is divergence free.
\end{proof}
We investigate now the commutation properties between the operators $(\bin )_{i \in \zm}$ and the average operator. Since the left hand side of the Boltzmann equation \eqref{Equ6} can be written into conservation form, it is worth analyzing how to average the divergence of a vector field.
\begin{pro}
\label{DivAve} With the previous notations, we have for any smooth vector field  $\xi (y) = (\xi _1 (y), ...,\xi _m (y))$
\[
\ave{\Divy \xi } = \Divy \left \{\sumio \ave{\xi \cdot \nabla _y \psi _i } b^i \right \} = \Divy \left \{\frac{\ave{\xi \dny \psi _0}}{\bzn \psi _0} b^0  + \sumio \ave{\xi \cdot \nabla _y \psi _i } b^i \right \}.
\]
\end{pro}
\begin{proof}
Notice that $\Divy \left \{\frac{\ave{\xi \dny \psi _0}}{\bzn \psi _0} b^0\right \} = 0$ and therefore it is enough to prove only the first equality. Obviously, the left hand side of the first equality  is constant along the flow of $b^0$. The same happens for the right hand side since
\[
\Divy\left \{\sumio \ave{\xi \dny \psi _i} b^i  \right \} = \sumio{\ave{\xi \dny \psi _i}\Divy b^i} + \sumio{\bin \ave{\xi \dny \psi _i }}
\]
and we know that $\Divy b^i \in \ker(\bzn)$ and that $(\bin)_{i \in \om}$ leave invariant $\ker ( \bzn)$. Therefore, in order to prove the first equality, it is sufficient to test against smooth functions $\varphi \in \ker ( \bzn )$. By the definition of the fields $(b^i)_{i \in \om}$ we have
\[
\xi = \frac{\xi \dny \psi _0}{I} b^0 + \sumio (\xi \dny \psi _i ) b^i
\]
and thus for any smooth function $\varphi \in \ker (\bzn)$ we obtain
\[
\xi \dny \varphi = \sumio (\xi \dny \psi _i) (\bin \varphi).
\]
After integration by parts one gets (by taking into account that $\varphi, \bin \varphi \in \ker (\bzn)$)
\begin{eqnarray}
\inty{\varphi (y ) \ave{\Divy \xi }} & = & \inty{\varphi (y ) \Divy \xi } \nonumber \\
& = & - \inty{\xi \dny \varphi}\nonumber \\
& = & - \sumio \inty{(\xi \dny \psi _i) (\bin \varphi )} \nonumber \\
& = & - \sumio \inty{\ave{\xi \dny \psi _i} (\bin \varphi )} \nonumber \\
& = & \inty{\varphi (y) \Divy \left \{ \sumio \ave{\xi \dny \psi _i} b^i\right \}}\nonumber
\end{eqnarray}
saying that 
\[
\ave{\Divy \xi } = \Divy \left \{ \sumio \ave{\xi \dny \psi _i} b^i\right \}.
\]
\end{proof}
\begin{coro}
\label{InvFields} Assume that the fields $(b^i)_{i \in \om}$ are smooth and have bounded divergence. Then the operators $(\bin)_{i \in \zm}$ are commuting with the average operator.
\end{coro}
\begin{proof}
Obviously $\bzn $ and $\ave{\cdot}$ are commuting, since for any function $u \in \D(\bzn)$ we have
\[
\ave{\bzn u } = \bzn \ave{u} = 0.
\]
For any $i \in \om$ and function $u \in \D(\bin)$ we obtain thanks to Proposition \ref{DivAve} applied to the field $u b^i$
\[
\ave{\Divy (u b^i)} = \Divy ( \ave{u}b^i)
\]
and our conclusion follows easily, since we know that $\Divy b^i \in \ker (\bzn)$.
\end{proof}
\begin{remark}
\label{CasPartDiv} In the particular case of the vector fields $b^0 = (0,0,\oc \;^\perp v)$, $b^1 (1,0,0,0)$, $b^2 = (0,1,0,0)$, $b^3 = (0,0,v/|v|)$ we obtain for any smooth vector field $\xi = (\xi _x (x,v), \xi _v (x,v))$ the formulae, thanks to the identity $\Divv \{v/|v|^2\} = 0$
\[
\ave{\Divxv \xi } = \Divx \ave{\xi _x} + \frac{v}{|v|^2}\cdot \nabla _v \ave{\xi _v \cdot v},\;\;\ave{\Divx \xi _x} = \ave{\Divxv (\xi _x, 0)} = \Divx \ave{\xi _x}.
\]
\end{remark}
The previous result guarantees that a transport operator reduces, after average, to another transport operator and allows us to keep the model into conservative form cf. \cite{BosGuidCent3D}. For the sake of transparency we detail this computation in the following proposition.
\begin{pro}
\label{Step1} Consider $a = a(y)$ a smooth divergence free vector field in $\R^m$. Then there is a divergence free vector field $A = A(y)$ such that for any smooth function $u \in \ker (\bzn) \cup \{\psi _0\}$
\begin{equation}
\label{Equ35} \ave{\Divy (ua) } = \Divy (uA).
\end{equation}
\end{pro}
\begin{proof}
We apply Proposition \ref{DivAve} to the field $y \to u(y) a(y)$
\[
\ave{\Divy ( ua)} = \Divy \left \{ \frac{\ave{ua \dny \psi _0}}{I} b^0 + \sumio \ave{ua \dny \psi _i } b^i   \right \}.
\]
If $u \in \ker (\bzn)$ then
\begin{equation}
\label{Equ36} \ave{\Divy ( ua)} = \Divy \left \{ u \frac{\ave{a \dny \psi _0}}{I} b^0 + \sumio u \ave{a \dny \psi _i } b^i   \right \}.
\end{equation}
In particular for $u = 1$ one gets
\begin{equation}
\label{Equ37} 0 = \ave{\Divy ( a)} = \Divy \left \{  \frac{\ave{a \dny \psi _0}}{I} b^0 + \sumio  \ave{a \dny \psi _i } b^i   \right \}.
\end{equation}
Hence, we define the field $A$ by
\begin{equation}\label{fieldA}
 A = \frac{\ave{a \dny \psi _0}}{I} b^0 + \sumio  \ave{a \dny \psi _i } b^i  
\end{equation}
so that $A$ is divergence free and the equation \eqref{Equ36} becomes
\[
\ave{\Divy ( ua)} = \Divy (uA),\;\;u \in \ker (\bzn).
\]
It remains to check that the same equality holds true for the angular coordinate $\psi _0$. Indeed, we have
\[
\ave{\Divy ( \psi _0 a)} =\ave{a \dny \psi _0} = A\dny \psi _0 = \Divy (\psi _0 A). 
\]
Actually the vector field $A$ constructed above is the only vector field satisfying \eqref{Equ35}, since necessarily  satisfies $A\dny \psi _i = \ave{a \dny \psi _i}$ for any $i \in \zm$.
\end{proof}
\begin{remark}
\label{CoordaA} If $(\alpha _0, ..., \alpha _{m-1})$ are the coordinates of $a$ with respect to the basis $\{b^0,...,b^{m-1}\}$, then $(\ave{\alpha _0}, ..., \ave{\alpha _{m-1}})$ are the coordinates of $A$ with respect to the same basis
\[
a = \sumiz \alpha _i b^i,\;\;A = \sumiz \ave{\alpha _i} b^i.
\]
\end{remark}
\begin{remark}
\label{CasPartaA} In the particular case of the vector fields $b^0 = (0,0,\oc \;^\perp v)$, $b^1 (1,0,0,0)$, $b^2 = (0,1,0,0)$, $b^3 = (0,0,v/|v|)$, $a = (v,q/m E)$ we obtain the coefficients
\[
\alpha _0 = \frac{qE}{m\oc}\cdot \frac{^\perp v}{|v|^2},\;\;\alpha _1 = v_1,\;\;\alpha _2 = v_2,\;\;\alpha _3 = \frac{qE}{m} \cdot \frac{v}{|v|}.
\]
It is easily seen that $\ave{\alpha _i} = 0, i \in \{0,...,3\}$ and thus ${\cal A} = \sum _{i = 0} ^3 \ave{\alpha _i} b^i = 0$.
\end{remark}
\bigskip
Notice that for any $i \in \zm$ the function $\alpha _i - \ave{\alpha _i}$ has zero average, and thus, by Proposition \ref{TransportProp}, there is a unique $\beta _i $ satisfying $\bzn \beta _i = \alpha _i - \ave{\alpha _i}$, with $\ave{\beta _i } = 0$. Therefore, we denote by $B$ the vector field
\begin{equation}
\label{fieldB}
 B = \sumiz \beta _i b^i
\end{equation}
whose characterization makes the object of the next proposition.
\medskip
\begin{pro}
\label{Step2} Consider $a = a(y)$ a smooth, divergence free vector field in $\R^m$. Let $A$ and $B$ be the vector fields defined in \eqref{fieldA} and  \eqref{fieldB} respectively. \\
i) Then for any smooth function $u \in \ker (\bzn) \cup \{\psi _0\}$ we have
\begin{equation}
\label{Equ39} \Divy \{(B \dny u ) b^0   \} = \Divy \{u (a-A)  \}
\end{equation}
and $\Divy (IB) = 0$.\\
ii) Let $\tilde{B} = B + \lambda _0 b^0$ where $\bzn \lambda _0 = \frac{B \dny I}{I}$, $\ave{\lambda _0} = 0$. Then for any smooth function $u \in \ker (\bzn)$ we have
\begin{equation}
\label{Equ39Bis}
\Divy \{(\tilde{B} \dny u ) b^0   \} = \Divy \{u (a-A)  \},\;\;\Divy \tilde{B} = 0.
\end{equation}
\end{pro} 
\begin{proof}
i) For any smooth function $u \in \ker (\bzn)$, taking into account that $(\bin)_{i\in \zm}$ leave invariant $\ker (\bzn)$, we have
\[
\Divy\{(B \dny u )b^0  \} =  \bzn (B\dny u) = \sumiz (\bin u) (\bzn \beta _i) = \Divy \{u(a-A)  \}.
\]
Taking now $u = \psi _0$ one gets
\[
\bzn (B \dny \psi _0) = \bzn (\beta _0 I) = I \bzn \beta _0 = I(\alpha _0 - \ave{\alpha _0}) = (a-A) \dny \psi _0.
\]
It remains to establish that $IB$ is divergence free. Notice that for any smooth function $u \in \ker (\bzn)$, the equation \eqref{Equ39} becomes
\[
[b^0,B]\dny u = \bzn (B \dny u) - B \dny (\bzn u) = (a-A) \dny u.
\]
In particular taking $u \in \{\psi _1,...,\psi _{m-1}\}$ we deduce that
\[
\{  [B, b^0] + a-A \}  \perp \mathrm{span} \{ \nabla _y \psi _1,...,\nabla _y \psi _{m-1} \} = (\R b^0) ^{\perp}
\]
implying that 
\begin{equation}
\label{Equ40} [B, b^0] + a-A = \lambda b^0.
\end{equation}
Taking now $u = \psi _0$ we obtain
\[
B \dny (\bzn \psi _0) - \bzn (B \dny \psi _0)   + (a-A) \dny \psi _0 = \lambda I
\]
and since we know that $\bzn ( B \dny \psi _0) = (a-A) \dny \psi _0$ we deduce that 
\begin{equation}
\label{Equ41} B \dny I = \lambda I.
\end{equation}
Taking the divergence of \eqref{Equ40} implies $- \bzn \Divy B = \bzn \lambda$ and therefore 
\begin{equation}
\label{Equ42} \Divy B + \lambda = \mu \in \ker (\bzn).
\end{equation}
Combining \eqref{Equ41}, \eqref{Equ42} yields $\Divy (IB) = \mu I$, $\bzn \mu = 0$. We are done if we prove that $\Divy (IB)$ has zero average, since, in this case we write
\[
\mu I = \ave{\mu I} = \ave{\Divy (IB)} = 0.
\]
Indeed, by Proposition \ref{DivAve} we have
\[
\ave{\Divy (IB)} = \Divy \left \{ \sumio \ave{IB \dny \psi _i} b^i  \right \} = \Divy \left \{ \sumio \ave{I \beta _i} b^i \right \} = 0.
\]
ii) Obviously we have $\Divy \tilde{B} = \Divy B + \bzn \lambda _0 = \Divy(IB)/I = 0$ and for any smooth function $u \in \ker (\bzn)$ we can write
\[
\Divy \{(\tilde B \dny u ) b^0   \} = \Divy \{ ( B \dny u ) b^0 \}= \Divy \{u (a- A)   \}.
\]
\end{proof}
\begin{remark}
\label{InvB} For any smooth function $u \in \ker (\bzn ) \cup \{\psi _0\} $ we have $\ave{B \dny u } = \ave{\tilde{B} \dny u } = 0$.
\end{remark}
\begin{remark}
\label{BInvT} For any function $u \in \ker (b ^0 \cdot \nabla _y)$ we have by Proposition \ref{Step1}
\[
\ave{(a-A) \cdot \nabla _y u } = \ave{\Divy (ua)} - \ave{\Divy (u A) } = 0
\]
and therefore, by Proposition \ref{TransportProp}, we know that there is a unique zero average function $w$ solving $b^0 \cdot \nabla _y w = (a - A) \cdot \nabla _y u$. Proposition \ref{Step2} says that $w = \tilde{B} \cdot \nabla _y u = B \cdot \nabla _y u$.
\end{remark}
\begin{remark}
\label{CasPartB} In our particular case (see Remark \ref{CasPartaA}) we obtain the coefficients
\[
\beta _0 = \frac{qE }{m \omega _c ^2}\cdot \frac{v}{|v|^2},\;\;\beta _1 = - \frac{v_2}{\omega _c},\;\;\beta _2 = \frac{v_1}{\omega _c},\;\;\beta _3 = - \frac{qE}{m \omega _c} \cdot \frac{^\perp v}{|v|} = v _\wedge \cdot \frac{v}{|v|}
\]
where $v_\wedge = {^\perp E}/B$ is the electric cross field drift, and therefore
\[
{\cal B} = \sum _{i = 0} ^3 \beta _i b^i = \left (- \frac{^\perp v}{\omega _c}, v_\wedge   \right ),\;\;\Divxv {\cal B} = \frac{^\perp v \cdot \nabla _x \omega _c}{\omega _c ^2}
\]
\[
\lambda _0 = - \frac{v \cdot \nabla _x \omega _c}{\omega _c ^3 },\;\;\tcalB = \left (- \frac{^\perp v}{\omega _c}, v_\wedge - \frac{v \cdot \nabla _x \omega _c}{\omega _c ^2 }\;^\perp v  \right ).
\]
For any function $u(x,v) \in \ker {\cal T}$ the unique zero average function $w(x,v)$ solving 
\[
{\cal T}w = (a- {\cal A}) \cdot \nabla_{x,v} u = a\cdot \nabla _{x,v} u = v\cdot \nabla _x u + \frac{q}{m}E \cdot \nabla _v u
\]
is given by 
\[
w = \tilde{\cal B} \cdot \nabla _{x,v} u = {\cal B} \cdot \nabla _{x,v} u = - \frac{^\perp v}{\oc} \cdot \nabla _x u + v _\wedge \cdot \nabla _v u.
\]
\end{remark}
\bigskip
The last field we compute will be useful when deriving the second order approximation, in Section \ref{SecOrdMod}.
\begin{pro}
\label{Step3} Consider $a = a(y)$ a smooth, divergence free vector field in $\R^m$ and let $\tilde{B}$ be the vector field defined in Proposition \ref{Step2}, ii). Then there is a divergence free vector field $C = C(y)$ such that for any smooth function $u \in \ker (\bzn) \cup \{\psi _0 \}$ we have
\[
\ave{\Divy \{( \tilde{B} \dny u) a   \}} = \Divy ( uC).
\]
\end{pro}
\begin{proof}
For the field $B$ given by \eqref{fieldB} and for any smooth function $u \in \ker (\bzn)$ we have by Proposition \ref{DivAve} 
\begin{eqnarray}
\label{Equ44} \ave{\Divy \{( \tilde{B} \dny u) a   \}} & = & \Divy \left \{\sumio \ave{(B \dny u) ( \an \psi _i )}b^i  \right \} \nonumber \\
& = & \Divy \left \{ \sumio \sumjo \ave{\alpha _i \beta _j} (\bjn u) b^i  \right \} \nonumber \\
& = & \sumjo \Divy \left \{\sumio \ave{\alpha _i \beta _j} b^i  \right \} (\bjn u) \nonumber \\
& + &  \sumjo \sumio \ave{\alpha _i \beta _j} \bin (\bjn u).
\end{eqnarray}
The last term vanishes, as a contraction between the anti-symmetric matrix $(\ave{\alpha _i \beta _j})_{i,j}$ and the symmetric matrix $(\bin (\bjn u))_{i,j}$. Indeed, by Proposition \ref{IntByPart} we obtain
\[
\ave{\alpha _i \beta _j} = \ave{(\alpha _i - \ave{\alpha _i}) \beta _j} = \ave{(\bzn \beta _i) \beta _j} = - \ave{\beta _i (\bzn \beta _j)} = - \ave{\alpha _j \beta _i}
\]
and by Proposition \ref{ComFields} one gets
\[
\bin (\bjn u) - \bjn (\bin u) = [b^i,b^j] \dny u = 0,\;\;i,j \in \om.
\]
Using again Proposition \ref{DivAve}, we deduce that for any $j \in \om$
\begin{equation}
\label{Equ45} \Divy \left \{ \sumio \ave{\alpha _i \beta _j } b^i  \right \} = \Divy \left \{ \sumio   \ave{\beta _j \an \psi _i } b^i \right \} = \ave{\Divy (\beta _j a)}.
\end{equation}
Combining \eqref{Equ44}, \eqref{Equ45} we obtain
\[
\ave{\Divy \{(\tilde{B} \dny u) a \}} = \sumjo \gamma _j \bjn u = \sumjz \gamma _j \bjn u
\]
with the notation
\[
\gamma _0 = \frac{\ave{\Divy (\beta _0 + \lambda _0 ) I a)}}{I},\;\;\gamma _j = \ave{\Divy (\beta _j a)},\;\;j \in \om,\;\;C = \sumjz \gamma _j b^j.
\]
It is easily seen that the above formula still holds true for $u = \psi _0$. It remains to prove that $\Divy C = 0$. Since $\Divy (\gamma _0 b^0) = 0$, it is enough to prove that $\Divy \{\sumio \gamma _i b^i  \} = 0$. For any smooth function $u \in \ker ( \bzn)$ we have
\[
- \inty{\!\!u \;\Divy C} = \inty{\!\!C \dny u } = \inty{\!\!\ave{\Divy \{(B \dny u) a \}}} = \inty{\!\!\Divy \{(B \dny u) a \}} = 0
\]
and therefore $\ave{\Divy C} = 0$. But $\Divy C$ is constant along the flow of $b^0$
\[
\Divy C = \sumio \gamma _i \Divy b^i + \sumio \bin \gamma _i
\]
since $\gamma _i, \Divy b^i, \bin \gamma _i$ are constant along the same flow. Therefore $\Divy C = \ave{\Divy C} = 0$.
\end{proof}
\begin{remark}
\label{ZeroCoeff} Actually we will need a divergence free vector field $C$ satisfying the statement in Proposition \ref{Step3} only for functions $u \in \ker (\bzn)$. The coordinates along $(b^i)_{i \in \om}$ of such vector fields are uniquely determined by $\gamma _i = \ave{\Divy (\beta _i a)}$, $i \in \om$ and the coordinate along $b^0$ could be any smooth function in $\ker (\bzn)$, which guarantees the constraint $\Divy C = 0$.
\end{remark}
\begin{remark}
\label{FieldC}  In the particular case of the vector fields $b^0 = (0,0,\oc \;^\perp v)$, $b^1 (1,0,0,0)$, $b^2 = (0,1,0,0)$, $b^3 = (0,0,v/|v|)$, $a = (v,q/m E)$ we obtain the coefficients
\[
(\gamma _1, \gamma _2) = - v_\wedge - v_{\mathrm{GD}},\;\;\gamma _3 = - \frac{|v| v _\wedge }{2} \cdot \frac{\nabla _x B}{B}
\]
where $v_\wedge, v_{\mathrm{GD}}$ are the (rescaled) electric cross field drift and magnetic gradient drift
\[
v_\wedge = \frac{^\perp E}{B},\;\;v_{\mathrm{GD}} = - \frac{|v|^2}{2\oc} \frac{\;^\perp \nabla _x B}{B}
\]
and the vector field ${\cal C}$ becomes
\begin{equation}
\label{EquCComponent}
{\cal C} = \left ( - v_\wedge - v_{\mathrm{GD}}, - \frac{v_\wedge \cdot \nabla _x B}{2B}v + \gamma _0 \oc \;^\perp v  \right)
\end{equation}
for some $\gamma _0 \in \ker {\cal T}$. It is easily seen that $\Divx v_\wedge = - (v _\wedge \cdot \nabla _x B)/B$, $\Divx v_{\mathrm{GD}} = 0$ and thus $\Divxv {\cal C} = 0$. In the sequel we will consider $\gamma _0 = 0$.
\end{remark}
\begin{pro}
\label{Step4} Consider $a = a(y)$ a smooth, divergence free vector field in $\R^m$. Then for any smooth function $u$ we have
\[
\Divy \{(\tilde{B} \dny u) A   \} = \Divy \{ u [A, \tilde{B}]\} + \Divy \{(A\dny u ) \tilde{B}\}
\]
where the vector fields $A$ and $\tilde{B}$ are given by \eqref{fieldA} and Proposition \ref{Step2}, ii), respectively. 
\end{pro}
\begin{proof}
The field $[A,\tB]$ is divergence free
\[
\Divy [A,\tB] = A \dny \Divy \tB - \tB \dny \Divy A = 0
\]
and therefore 
\begin{eqnarray}
\Divy\{(\tB \dny u ) A  \} & = &  A \dny ( \tB \dny u)  \nonumber \\
& = &  [A, \tB] \dny u + \tB \dny (A \dny u) \nonumber \\
& = &  \Divy \{ u [A, \tB]\} + \Divy \{(A\dny u ) \tilde{B}\}. \nonumber 
\end{eqnarray}
\end{proof}
\begin{remark}
\label{Use} Let $u \in \ker (\bzn )$ and $z$ satisfying
\[
\Divy \{ u (a - A)\} + \Divy \{z b^0\} = 0,\;\;\ave{z} = 0
\]
which means that $z = - \tB \dny u$, cf. Proposition \ref{Step2}. Then combining Propositions \ref{Step3} and \ref{Step4} we obtain
\begin{eqnarray}
\ave{\Divy (za)} - \Divy (zA) & = & - \ave{\Divy \{(\tB \dny u)a   \}} + \Divy \{(\tB \dny u )A  \} \nonumber \\
& = & - \Divy (uC) + \Divy \{u [A, \tB]  \} + \Divy \{ (A \dny u ) \tB \} \nonumber \\
& = & \Divy \{u ( [A, \tB] - C) + (A \dny u ) \tB   \}.\nonumber 
\end{eqnarray}
\end{remark}

\subsection{Average and the collision operator}
\label{AveColOpe}
\indent

In this section we will show that the average operator associated to ${\cal T} = \oc {^\perp v} \cdot \nabla _v$ and the collision  operator are commuting. This means that at the lowest order the collision mechanism is not affected by strong magnetic fields. Corrections appear only at the next orders, as commutators between transport and convolution through the scattering cross section, see Proposition \ref{CollCommB}. We first recall the classical properties of the linear Boltzmann operator $Q$. The notation $\ltm$ stands for $\ltmvx$.

\begin{pro}
\label{LinBolOpe} Assume that the scattering cross section satisfies \eqref{Equ4}. Then the gain/loss collision operators $Q_\pm$
\[
Q_+ (f)(x,v) = \frac{1}{\tau} \intvp{s(v,\vp) M(v) f(x,\vp)},\;Q_- (f)(x,v) = \frac{1}{\tau} \intvp{s(v,\vp) M(\vp) f(x,v)}
\]
are linear bounded operators from $\ltm$ to $\ltm$ and $\|Q_\pm\| \leq S_0 /\tau$. Moreover, if the scattering cross section also satisfies \eqref{Equ3}, then $Q_\pm$ are symmetric with respect to the scalar product of $\ltm$.
\end{pro}
Notice that $M$ depends only on $|v|$ and therefore $M \in \ker {\cal T}$, ${\cal T} = \oc \;^\perp v \cdot \nabla _v $. Observe that the average operator is also bounded from $\ltm$ to $\ltm$. Indeed, for any $f \in \ltm$ we have $f/\sqrt{M} \in L^2 (\mathrm{d}v \mathrm{d}x)$ and thus
\[
\|\ave{f} \|^2 _{L^2_M} = \left \| \ave{\frac{f}{\sqrt{M}}} \right \|^2 _{L^2} \leq \left \| \frac{f}{\sqrt{M}} \right \|^2 _{L^2} = \|f \|^2 _{L^2_M}. 
\]
We assume that  there is a function $\sigma : \R_+ \to [s_0,S_0]$ such that 
\begin{equation}
\label{Equ47} s(v,\vp) = \sigma ( |v - \vp|),\;\;v, \vp \in \R.
\end{equation}
\begin{lemma}
\label{LC} Assume that the scattering cross section satisfies \eqref{Equ47}. Then the gain/loss collision operators are commuting with the characteristic flow of ${\cal T}$. In particular the gain/loss collision operators leave invariant $\ltm \cap \ker {\cal T}$ and $\ltm \cap \ker \ave{\cdot}$. 
\end{lemma}
\begin{proof}
Consider $f \in \ltm $ and for any $\alpha \in \R$ let us introduce $f_\alpha (x,v) = f(x, R_\alpha v)$, where $R(\alpha)$ stands for the rotation of angle $\alpha$ (see \eqref{EquFlow}). Obviously $f _\alpha \in \ltm$ and
\begin{eqnarray}
(Q_+ f_\alpha ) (x,v) & = & \frac{1}{\tau} \intvp{\sigma ( | v - \vp|)M(v) f _\alpha (x,\vp)} \nonumber \\
& = & \frac{1}{\tau} \intvp{\sigma ( | v - \vp|)M(v) f(x,R_\alpha \vp)} \nonumber \\
& = & \frac{1}{\tau} \int_{\R^2} \sigma ( | v - R_{-\alpha} w^\prime |)M(v) f (x,w^\prime) \;\mathrm{d}w^\prime \nonumber \\
& = & \frac{1}{\tau} \int_{\R^2}\sigma ( | R_\alpha v - w^\prime|)M(R_\alpha v) f  (x,w^\prime) \;\mathrm{d}w^\prime\nonumber \\
& = & (Q_+f)(x,R_\alpha v).\nonumber 
\end{eqnarray}
Similarly one gets
\begin{eqnarray}
(Q_- f_\alpha ) (x,v) & = & \frac{1}{\tau} \intvp{\sigma ( | v - \vp|)M(\vp) f _\alpha (x,v)} \nonumber \\
& = & \frac{1}{\tau} \intvp{\sigma ( | v - \vp|)M(\vp) f(x,R_\alpha v)} \nonumber \\
& = & \frac{1}{\tau} \int_{\R^2} \sigma ( | R_\alpha v - R_{\alpha} \vp |)M(R_\alpha \vp) f (x,R_\alpha v) \;\mathrm{d}\vp \nonumber \\
& = & \frac{1}{\tau} \int_{\R^2}\sigma ( | R_\alpha v - w^\prime|)M(w^\prime) f  (x,R_\alpha v) \;\mathrm{d}w^\prime\nonumber \\
& = & (Q_-f)(x,R_\alpha v).\nonumber 
\end{eqnarray}
In particular, if $f \in \ltm{} \cap \ker {\cal T}$ then
\[
(Q_\pm f)(x, R_\alpha v) = (Q_\pm f_\alpha )(x,v) = (Q_\pm f)(x,v)
\]
saying that $Q_\pm f \in \ker {\cal T}$. 
Consider now $h \in \ltm{} \cap \ker \ave{\cdot} $ and $\varphi  \in \ker {\cal T}$. Then $Q_\pm (h)/\sqrt{M} \in L^2(\mathrm{d}v\mathrm{d}x)$, $\sqrt{M} \varphi \in \ltm{}$ and we can write
\[
\intxvbis{\frac{Q_\pm (h)}{\sqrt{M}} \varphi } = ( Q_\pm (h), \sqrt{M} \varphi ) _\ltm  = (h, Q_\pm (\sqrt{M} \varphi )) _\ltm = \intxvbis{\frac{h}{\sqrt{M}} \frac{Q_\pm (\sqrt{M} \varphi )}{\sqrt{M}}}.
\]
By the previous computations we know that $Q_\pm (\sqrt{M} \varphi )\in \ltm \cap \ker {\cal T}$ and thus $Q_\pm (\sqrt{M} \varphi )/\sqrt{M} \in L^2 \cap \ker {\cal T}$. Since $h/\sqrt{M} \in L^2 \cap \ker \ave{\cdot}$ we deduce that 
\[
\intxv{\frac{Q_\pm (h)}{\sqrt{M}} \varphi } = \intxv{\frac{h}{\sqrt{M}} \frac{Q_\pm (\sqrt{M} \varphi )}{\sqrt{M}}} = 0
\]
saying that $\ave{Q_\pm (h) /\sqrt{M} } = 0$ and thus $Q_\pm (h) \in \ltm \cap \ker \ave{\cdot}$.
\end{proof}
\begin{coro}
\label{ColAve} The average and the gain/loss collision operators are commuting on $\ltm{}$. 
\end{coro}
\begin{proof}
Consider $f \in \ltm{}$. Using the decomposition $f = \ave{f} + (f - \ave{f})$ one gets by the previous lemma that ${\cal T} Q_\pm \ave{f} = 0$, $\ave{Q_\pm (f - \ave{f})} = 0$ and therefore
\[
\ave{Q_\pm (f) } = \ave{Q_\pm \ave{f}} + \ave{Q_\pm (f - \ave{f})} = Q_\pm \ave{f}.
\]
\end{proof}
We end this section with the following commutation result, which will be used when investigating the second order approximation.

\begin{pro}
\label{CollCommB}
Assume that the scattering cross section satisfies \eqref{Equ47} and it is smooth. Then for any smooth function $f = f(x,v) \in \ker {\cal T}$ we have
\begin{eqnarray}
\tQo (f) & := &  Q(\tcalB \cdot \nxv f ) - \tcalB \cdot \nxv Q(f) \nonumber \\
& = & \frac{1}{\tau} \intvp{\left ( s_1 - \frac{|\vp|}{|v|} s_2 \right ) M(v) \frac{^\perp v}{\oc} \cdot \nabla _x f (x,\vp)} \nonumber \\
&+ & \frac{m}{\tau \theta}(v \cdot v_\wedge) \intvp{s_1 M(v) f(x,\vp)} 
 -  \frac{m}{\tau \theta}(v \cdot v_\wedge) \intvp{\frac{|\vp|}{|v|}s_2 M(\vp) f(x,v)} \nonumber 
\end{eqnarray}
where (see Remark \ref{CasPartB})
\[
\tcalB = \left (- \frac{^\perp v}{\omega _c}, v_\wedge - \frac{v \cdot \nabla _x \omega _c}{\omega _c ^2 }\;^\perp v  \right ),\;\;s_1(v,\vp) = \ave{s(v,\vp)},\;\;s_2(v,\vp) = \ave{s(v,\vp)\frac{v \cdot \vp}{|v|\;|\vp|}}.
\]
\end{pro}
\begin{proof}
Observe that $\nv s + \nvp s = 0$ and therefore, performing integration by parts with respect to $\vp$ yields
\begin{eqnarray}
\tau \tcalB \cdot \nxv Q(f) & = & \intvp{s(v,\vp)\{ M(v) \tcalB (x,v) \cdot \nxvp f(x,\vp) - M(\vp) \tcalB (x,v) \cdot \nxv f(x,v)   \}} \nonumber \\
& + & \tcalB _v (x,v) \cdot \intvp{\{\nv M(v) f(x,\vp) - \nvp M(\vp) f(x,v)\} s(v,\vp)}.\nonumber
\end{eqnarray}
Combining with the equality
\[
\tau Q(\tcalB \cdot \nxv f)  =  \intvp{s(v,\vp)\{M(v) \tcalB(x,\vp) \cdot \nxvp f(x,\vp) - M(\vp)\tcalB (x,v) \cdot \nxv f(x,v)\}}
\]
we deduce that 
\begin{eqnarray}
\label{EquConvert} \tau \tcalB \cdot \nxv Q(f) & - &  \tau Q(\tcalB \cdot \nxv f ) = \intvp{s M(v) ( \tcalB (x,v) - \tcalB (x, \vp)) \cdot \nxvp f (x, \vp)} \nonumber \\
& + & \tcalB _v (x,v) \cdot \intvp{\{\nv M(v) f(x,\vp) - \nvp M(\vp) f(x,v)\} s}.
\end{eqnarray}
We transform the above terms appealing to the symmetry of $f$. Indeed, we have
\begin{eqnarray}
\label{EquConvert1} T_1 & := & \intvp{s(v, \vp) M(v) ( \tcalB _x (x,v) - \tcalB _x (x, \vp)) \cdot \nx f (x, \vp)} \nonumber \\
& = & \intvp{s(v,\vp) M(v)  \frac{^\perp ( \vp - v)}{\oc}  \cdot \nx f (x, \vp)} \nonumber \\ 
& = & \intvp{\frac{M(v)}{\oc} \ave{s(v,\vp) {\;^\perp ( \vp - v)}}   \cdot \nx f (x, \vp)}.
\end{eqnarray}
We introduce the following averaged (with respect to $\vp$) scattering cross sections
\[
s_1(v,\vp) = \ave{s(v,\vp)},\;\;s_2(v,\vp) = \ave{s(v,\vp)\frac{v \cdot \vp}{|v|\;|\vp|}}.
\]
Writing that $\vp = \rp \cos \alpha \;v/|v| - \rp \sin \alpha \;^\perp v /|v|, \rp = |\vp|, r = |v|$ one gets
\[
\vp - v = (\rp \cos \alpha - r ) \frac{v}{|v|} - \rp \sin \alpha \;\frac{^\perp v}{|v|}
\]
\[
s_1 (v,\vp) = \frac{1}{2\pi } \int _0 ^{2\pi} \sigma ( \sqrt{(\rp )^2 + r^2 - 2 r \rp \cos \alpha })\;\mathrm{d}\alpha
\]
\[
s_2 (v,\vp) = \frac{1}{2\pi } \int _0 ^{2\pi} \sigma ( \sqrt{(\rp )^2 + r^2 - 2 r \rp \cos \alpha })\cos \alpha \;\mathrm{d}\alpha
\]
saying that $s_1, s_2$ depend only on $r = |v|, \rp = |\vp|$ and that they are symmetric with respect to $r, \rp$. Notice also that 
\[
\ave{s(v,\vp)\vp} = \frac{1}{2\pi } \int _0 ^{2\pi} \!\!\!\!\!\!\!\sigma ( \sqrt{(\rp )^2 + r^2 - 2 r \rp \cos \alpha })\left [ \rp \cos \alpha \frac{v}{|v|} - \rp \sin \alpha \frac{^\perp v}{|v|} \right ]\;\mathrm{d}\alpha = s_2 |\vp| \frac{v}{|v|}
\]
and
\[
\ave{s(v,\vp) (v - \vp)} = \left ( s_1 (v,\vp) - \frac{|\vp|}{|v|} s_2 (v,\vp) \right ) v.
\]
Therefore, the integral on the last line of \eqref{EquConvert1} writes
\begin{equation}
\label{EquConvert2} T_1 = - \intvp{\left ( s_1 (v,\vp) - \frac{|\vp|}{|v|} s_2 (v,\vp) \right )M(v) \frac{^\perp v}{\oc} \cdot \nabla _x f (x,\vp)}.
\end{equation}
By taking into account that $^\perp \vp \cdot \nvp f(x,\vp) = 0$, we observe that
\begin{eqnarray}
\label{EquConvert3} T_2 & := & \intvp{s(v,\vp) M(v) ( \tcalB _v (x,v) - \tcalB _v (x, \vp) ) \cdot \nvp f(x,\vp) } \nonumber \\
& = & \intvp{s(v,\vp) M(v) \frac{v \cdot \nabla _x \omega _c }{\omega _c ^2}\;^\perp ( \vp - v ) \cdot \nvp f(x,\vp) } \nonumber \\
& = & \intvp{s(v,\vp) M(v) \frac{v \cdot \nabla _x \omega _c }{\omega _c ^2}\;\mathrm{div}_{\vp} \{^\perp ( \vp - v ) f(x,\vp) \}} \nonumber \\
& = & - \intvp{\sigma ^\prime (|v - \vp|) \frac{\vp - v}{|\vp - v|} \cdot \;^\perp ( \vp - v ) f(x,\vp)M(v) \frac{v \cdot \nabla _x \omega _c }{\omega _c ^2}} = 0.
\end{eqnarray}
For the last integral in \eqref{EquConvert} we notice that 
\begin{equation}
\label{EquConvert4} T_3 := \tcalB _v (x,v) \cdot \intvp{\nabla _v M(v) f(x,\vp) s(v,\vp)} = - \frac{m}{\theta} (v \cdot v_\wedge) \intvp{s_1 (v,\vp) M(v) f(x,\vp)}
\end{equation}
and
\begin{eqnarray}
\label{EquConvert5}T_4 & := & - \tcalB _v (x,v) \cdot \intvp{\nvp M(\vp) f(x,v) s(v,\vp)} \nonumber \\
& = & \tcalB _v (x,v) \cdot \intvp{\frac{m}{\theta} \ave{s(v,\vp)\vp} M(\vp) f(x,v) } \nonumber \\
& = & \frac{m}{\theta} (v_\wedge \cdot v) \intvp{s_2(v,\vp) \frac{|\vp|}{|v|} M(\vp) f(x,v) }.
\end{eqnarray}
Finally combining \eqref{EquConvert}, \eqref{EquConvert2}, \eqref{EquConvert3}, \eqref{EquConvert4}, \eqref{EquConvert5} we obtain
\begin{eqnarray}
\tQo (f) & = &   \frac{1}{\tau} \intvp{\left ( s_1 (v,\vp) - \frac{|\vp|}{|v|} s_2 (v,\vp) \right )M(v) v \cdot \left \{  \frac{m}{\theta} v_\wedge f(x,\vp) - \frac{^\perp \nabla _x f (x,\vp)}{\oc} \right \}} \nonumber \\
& + & \frac{1}{\tau} \intvp{s_2 (v,\vp)  \frac{|\vp|}{|v|} \frac{m}{\theta} (v \cdot v_\wedge) \{ M(v) f(x,\vp) - M(\vp) f(x,v) \}} \nonumber \\
& = & \frac{1}{\tau} \intvp{\left ( s_1 (v,\vp) - \frac{|\vp|}{|v|} s_2 (v,\vp) \right )M(v) F(x,\vp)  \frac{^\perp v}{\oc} \cdot \nabla _x \left (\frac{f(x,\vp)}{F(x,\vp)}   \right )} \nonumber \\
& + &\frac{1}{\tau} \intvp{s_2 (v,\vp)  \frac{|\vp|}{|v|} \frac{m}{\theta} (v \cdot v_\wedge) M(v)F(x,\vp)\left \{ \frac{f(x,\vp)}{F(x,\vp)} - \frac{ f(x,v)}{F(x,v)} \right \}}\nonumber
\end{eqnarray}
where $F(x,v) = M(v) \exp ( - q \phi (x) /\theta)$ is the equilibrium of \eqref{Equ6}, $E(x) = - \nabla _x \phi$.
\end{proof}
Notice that the operator $\tQo$ has the following properties\\
\begin{enumerate}
\item the equilibrium $F$ belongs to the kernel of $\tQo$, since $\tilde{{\cal B}} \cdot \nxv F = 0$ and $Q(F) = 0$
\[
\tQo (F) = Q(\tilde{{\cal B}} \cdot \nxv F) - \tilde{{\cal B}} \cdot \nxv Q(F) = 0.
\]
\item $\tQo$ maps constant functions along the flow of $b^0$ to zero average functions since, for any smooth function $f \in \ker {\cal T}$, $\tilde{{\cal B}} \cdot \nxv  f \in \ker \ave{\cdot}$, $Q(f) \in \ker {\cal T}$ cf. Lemma \ref{LC} and therefore
\[
Q(\tcalB \cdot \nxv f) \in \ker \ave{\cdot},\;\;\tilde{{\cal B}} \cdot \nxv Q(f) \in \ker \ave{\cdot}.
\] 
\item In particular $\tQo$ satisfies the local mass conservation {\it i.e.,} for any smooth function $f \in \ker {\cal T}$ we have $\tQo (f) \in \ker \ave{\cdot} $ and thus
\[
\intv{\tQo (f) } = \intv{\ave{\tQo (f)}} = 0.
\]
\item $\tQo$ is symmetric with respect to the scalar product of $\ltF = \ltFvx$. Indeed, for any smooth $f,g \in \ker {\cal T}$ we have
\begin{eqnarray}
(\tQo (f),g) _\ltF & = & \intxv{Q(\tcalB \cdot \nxv f) \frac{g}{F}} - \intxv{\tcalB \cdot \nxv Q(f) \frac{g}{F}} \nonumber \\
& = & (Q(\tcalB \cdot \nxv f),g)_\ltF + (Q(f), \tcalB \cdot \nxv g)_\ltF \nonumber \\
& = & (\tcalB \cdot \nxv f, Q(g))_\ltF + (Q(f), \tcalB \cdot \nxv g) _\ltF \nonumber \\
& = & (\tQo (g), f) _\ltF.\nonumber
\end{eqnarray}
\item For any smooth $f \in \ker {\cal T}$ we have $(\tQo (f), f)_\ltF = 0$ since $\tQo (f) \in \ker \ave{\cdot}$, $f/F \in \ker {\cal T}$.
\end{enumerate}

\begin{remark}
When the scattering cross section is constant, we have $s_1 = s, s_2 = 0$ and the expression of $\tQo$ becomes
\[
\tQo (f) = \frac{s}{\tau} M(v) \rho _F (x) \frac{^\perp v}{\oc} \cdot \nabla _x \left (\frac{\rho _f}{\rho _F}  \right )
\]
where $\rho _f (x)= \intvp{f(x,\vp)}$, $\rho _F (x)= \intvp{F(x,\vp)} = \exp ( - q \phi (x) /\theta)$.
\end{remark}

\section{Asymptotic models for Boltzmann equation}
\label{LimBolEqu}
\noindent

This section is devoted to the derivation of the limit model for \eqref{Equ6} when $\eps $ becomes very small. At the formal level, we perform our analysis starting from the expansion $\fe = f + \eps f^1 + \eps ^2 f^2 + ...$. We are interested in first order but also second order asymptotic models. We present weak and strong convergence results.

\subsection{First order model}
\label{FirOrdMod}
\noindent

We investigate the asymptotic behaviour of the linear Boltzmann problem \eqref{Equ6}, \eqref{Equ2} when $\eps \searrow 0 $. At least formally, the leading order term $f$ satisfies \eqref{Equ7} and the constraint ${\cal T}f(t) = 0, t \in \R_+$. By Proposition \ref{TransportProp} we know that $\ran {\cal T} = \ker \ave{\cdot}$ and therefore \eqref{Equ7} is equivalent to
\[
\ave{\partial _t f + v \cdot \nabla _x f + \frac{q}{m} E \cdot \nabla _v f} = \ave{Q(f)}.
\]
It is easily seen that $\ave{\partial _t f } = \partial _t \ave{f} = \partial _t f$ and by Corollary \ref{ColAve} we know that $\ave{Q(f)} = Q(\ave{f}) = Q(f)$. Moreover, there is a function $g$ such that $f(t,x,v) = g(t, x, r = |v|)$ and therefore
\[
\ave{v \cdot \nabla _x f } = \ave{v} \cdot \nabla _x g = 0,\;\;\ave{E\cdot \nabla _v f } = \ave{E \cdot \frac{v}{|v|} \;\partial _r g } = \frac{\ave{v}}{|v|} \cdot E \;\partial _r g = 0
\]
since $\ave{v} = 0$.
We obtain the limit model
\[
\partial _t f = Q(f),\;\;{\cal T } f = 0
\]
which can be justified rigorously at least under the confinement potential hypothesis
\begin{equation}
E(x) = - \nabla _x \phi (x),\;\;\intx{\exp \left ( - q\frac{\phi (x)}{\theta} \right ) } <+\infty. \nonumber 
\end{equation}
Indeed, in this case $F(x,v) = M(v) \exp \left ( - q\phi (x)/\theta \right ) $ belongs to $\ltF$ and solves \eqref{Equ6} for any $\eps >0$. Notice that $Q_\pm$ are also bounded from $\ltF$ to $\ltF$. It is well know that for any convex function $H$, the relative entropy 
\[
{\cal H} ^\eps (t) = \intxv{F(x,v) H \left (\frac{\fe (t,x,v)}{F(x,v)}   \right )}
\]
satisfies
\begin{equation}
\label{Equ50} \frac{d{\cal H} ^\eps}{dt}  + {\cal D} ^\eps (t) = 0,\;\;t \in \R_+
\end{equation}
where the dissipative term ${\cal D} ^\eps$ is given by
\begin{equation}
\label{Equ51} {\cal D} ^\eps (t)  =  \frac{1}{\tau} \intxv{\intvp{s(v,\vp) M(v) F(x,\vp) D\left (\frac{\fe (t,x,v)}{F(x,v)}, \frac{\fe (t,x,\vp)}{F(x,\vp)}\right )}}
\end{equation}
where
\[
D(u, u^\prime) = H(u^\prime) - H(u) - (u^\prime - u) H^\prime (u) \geq 0,\;\;u, u^\prime \in \R. 
\]
Assuming that the initial condition satisfies 
\begin{equation}
\label{Equ52}
\fin \geq 0, \fin \in \ltF
\end{equation}
one gets by \eqref{Equ50} 
\[
\intxv{\frac{(\fe (t,x,v))^2}{F(x,v)}} \leq \intxv{\frac{(\fin (x,v))^2}{F(x,v)}},\;\;\eps >0
\]
which allows us to pass to the limit, at least weakly, in \eqref{Equ6}.
\begin{proof} (of Theorem \ref{WeakConv})
We have the uniform bound
\[
\supe \|\fe \| _\liltF \leq \|\fin \|_\ltF
\]
and we can assume (after extraction of a sequence $\epsk \searrow 0$)
\[
\limk \fek = f,\;\;\mbox{ weakly } \star \mbox{ in } \;\liltF.
\]
Using the weak formulation \eqref{Equ6} it is easily seen that the weak limit $f$ satisfies the constraint ${\cal T}f = 0$. Taking now test functions of the form $\eta (t) \varphi (x,v)$, $\eta \in C^1_c (\R_+)$, $\varphi \in C^1 _c ( \R^2 \times \R^2) \cap \ker {\cal T}$ (the constraint $\varphi \in \ker {\cal T}$ is imposed in order to remove the singular term $\frac{1}{\eps} {\cal T} \fe$ in \eqref{Equ6}) we obtain
\begin{eqnarray}
\label{Equ53} - \inttxv{\fek \left ( \eta ^\prime \varphi + \eta v \cdot \nabla _x \varphi + \eta \frac{q}{m} E \cdot \nabla _v \varphi \right ) } \!\!\!& -& \!\!\! \intxv{\fin \eta (0) \varphi }  \\
\!\!\!& = &\!\!\! \inttxv{Q(\fek) \eta \varphi }.\nonumber
\end{eqnarray}
By the symmetry of the collision operator one gets, noticing that $\eta Q(F\varphi ) \in L^1(\R_+;\ltF)$
\begin{eqnarray}
\limk \inttxv{Q(\fek)\eta \varphi } \!\!\!& = &\!\!\! \limk \int _{\R_+} (\fek (t),\eta (t) Q(F \varphi))_{\ltF} \;\mathrm{d}t \nonumber \\
\!\!\!& = &\!\!\! \int _{\R_+} (f (t),\eta (t) Q(F \varphi))_{\ltF} \;\mathrm{d}t \nonumber \\
\!\!\!& = &\!\!\! \inttxv{Q(f)\eta \varphi }.\nonumber
\end{eqnarray}
Passing to the limit with respect to $k \to +\infty$ in \eqref{Equ53} leads to the weak formulation
\begin{eqnarray}
\label{Equ54} - \inttxv{f \left ( \eta ^\prime \varphi + \eta v \cdot \nabla _x \varphi + \eta \frac{q}{m} E \cdot \nabla _v \varphi \right ) } \!\!\!& -& \!\!\! \intxv{\fin \eta (0) \varphi }  \\
\!\!\!& = &\!\!\! \inttxv{Q(f) \eta \varphi }\nonumber
\end{eqnarray}
for $\eta \in C^1 _c(\R_+)$, $\varphi \in C^1 _c(\R^2 \times \R^2) \cap \ker {\cal T}$. Writing $\varphi (x,v) = \psi (x, |v|)$ we deduce that $\ave{v \cdot \nabla _x \varphi } = 0$, $\ave{E \cdot \nabla _v \varphi } = 0$ and therefore, since ${\cal T}f = 0$, one gets
\[
\inttxv{f\left (  \eta v \cdot \nabla _x \varphi + \eta \frac{q}{m} E \cdot \nabla _v \varphi \right ) } = 0.
\]
Finally the formulation \eqref{Equ54} reduces to
\[
- \inttxv{f \partial _t (\eta \varphi) } - \intxv{\ave{\fin} \eta (0) \varphi } =  \inttxv{Q(f) \eta \varphi }\nonumber
\]
for test functions satisfying the constraint ${\cal T} (\eta \varphi ) = 0$ at any time $t \in \R_+$. Since $f, \ave{\fin}$ and $Q(f)$ also satisfy this constraint it is easily seen that $f$ solves weakly \eqref{Equ55}. By the uniqueness of the solution for \eqref{Equ55} we deduce that all the family $(\fe)_{\eps >0}$ converges weakly $\star$ in $\liltF$ towards $f$. It remains to justify that the constraint ${\cal T}f(t) =0$ is automatically satisfied at any time $t>0$. Indeed, for any $\alpha >0$ we have, cf. Lemma \ref{LC}
\[
\partial _t f_\alpha  = (\partial _t f) (x, R_\alpha v) = (Qf)(x,R_\alpha v) = (Qf_\alpha) (x,v).
\]
Since $f_\alpha (0,x,v) = \ave{\fin} (x,R_\alpha v) = \ave{\fin}(x,v) = f(0,x,v)$ we deduce that both $f$ and $f_\alpha$ solve \eqref{Equ55} and, by the uniqueness of the solution, one gets $f_\alpha (t) = f(t)$ for any $\alpha \in \R$, saying that ${\cal T} f(t) = 0$ at any time $t >0$.
\end{proof}
\bigskip
Under appropriate smoothness hypotheses on the initial condition $\fin$ and the electro-magnetic field, it is possible to obtain strong convergence results. Actually we can show that $\fe = f + {\cal O} (\eps)$ strongly, which justifies the first term in the expansion $\fe = f + \eps f^1 + ...\;$. We only sketch the main lines of the arguments and skip the standard details. The main idea is to introduce a first order correction term. Assume for the moment that the dominant term $f$ is smooth enough (with respect to all variables). Since
\[
\ave{\partial _t f + v \cdot \nabla _x f + \frac{q}{m} E \cdot \nabla _v f - Q(f) } = 0
\]
we deduce by Proposition \ref{TransportProp} that there is a unique function $f^1$ such that at any time $t>0$
\begin{equation}
\label{Equ59} \partial _t f + v \cdot \nabla _x f + \frac{q}{m} E \cdot \nabla _v f + {\cal T} f^1 =  Q(f),\;\;\ave{f^1} = 0.
\end{equation}
Combining with \eqref{Equ6} we obtain
\begin{equation}
\label{Equ57} 
\left (\partial _t  + v \cdot \nabla _x  + \frac{q}{m} E \cdot \nabla _v + \frac{\cal T}{\eps}  \right )(\fe -  f  - \eps f^1 )  = Q(\fe -  f  - \eps f^1 ) + \eps {\cal R}(f^1)
\end{equation}
where 
\begin{equation}
\label{Equ60}
{\cal R}(f^1) = -  \left (\partial _t  + v \cdot \nabla _x  + \frac{q}{m} E \cdot \nabla _v \right ) f^1 + Q(f^1).
\end{equation}
Notice that ${\cal R}(f^1)$ is of order of $1$, provided that $f^1$ is smooth enough (we will see later on how we can get regularity for $f^1$ from the regularity of $f$). Assume for the moment that $f^1$ is smooth enough. In this case $\eps {\cal R}(f^1)$ will be of order of $\eps$. Using again the relative entropy dissipation (similar to \eqref{Equ50}, since now we have an extra term $\eps {\cal R}(f^1)$) or multiplying \eqref{Equ57} by $(\fe - f - \eps f^1 )/F$, where $F$ is the equilibrium of \eqref{Equ6}, one gets
\begin{eqnarray}
\frac{1}{2}\frac{\mathrm{d}}{\mathrm{d}t} \| (\fe - f - \eps f^1)(t)\|^2 _\ltF & \leq & \eps \intxv{{\cal R }(f^1)  \frac{\fe - f - \eps f^1}{F}} \nonumber \\
& \leq & \eps \|{\cal R}(f^1 (t))\|_\ltF \| (\fe - f - \eps f^1)(t) \|_\ltF \nonumber
\end{eqnarray}
implying that 
\[
\frac{\mathrm{d}}{\mathrm{d}t} \| (\fe - f - \eps f^1)(t)\| _\ltF \leq \eps \|{\cal R}(f^1 (t))\|_\ltF. 
\]
Finally one gets
\[
\| (\fe - f - \eps f^1)(t)\| _\ltF \leq \|\fin - \ave{\fin} - \eps f^1 (0)\|_\ltF + \eps \int _0 ^t \|{\cal R} (f^1(s))\|_\ltF \;\mathrm{d}s
\]
and therefore, for any $T >0$, there is a constant $C_T$ depending on $\|f^1\|_\liltF$ and $\|{\cal R} (f^1)\|_{L^1(0,T;\ltF)}$ such that 
\[
\| (\fe - f)(t)\| _\ltF \leq C_T \eps,\;\;t \in [0,T],\;\;\eps >0
\]
provided that the initial condition is well prepared {\it i.e.,} ${\cal T} \fin = 0$. It remains to justify the smoothness of $f$ and $f^1$. 
\begin{pro}
\label{Smoothf} Assume that the initial condition satisfies $\fin, |v|\fin \in \ltF$, ${\cal T} \fin = 0$. Then the solution of \eqref{Equ55} satisfies $ f \in \liltF,\;\;vf  \in \liltF $.
\end{pro}
\begin{proof}
We know that $\|f(t) \|_\ltF \leq \|\fin \|_\ltF$. Multiplying the equation in \eqref{Equ55} by $v$ one gets
\[
\partial _t (v f) +v f(t,x,v) \frac{1}{\tau} \intvp{s(v,\vp) M(\vp)} =    v M(v) \frac{1}{\tau} \intvp{s(v,\vp) f(t,x,\vp)}.
\]
Solving the above differential equation with respect to $vf$ and estimating the right hand side by using
\[
\intvp{s(v,\vp) f(t,x,\vp)} \leq S_0 \;\|\fin \|_\ltm
\]
lead to the bound
\[
\|v f(t)\|_\ltF \leq \|v \fin \|_\ltF + {\frac{S_0}{s_0}} \left ( \intvp{|\vp |^2 M(\vp)}\right ) ^{1/2}  \|\fin \|_\ltF,\;\;t >0.
\]
\end{proof}
\begin{pro}
\label{SmoothfBis} Assume that the function $\sigma $, introduced in \eqref{Equ47}, is Lipschitz continuous and that the initial condition satisfies 
$$
{\cal T} \fin = 0,\;\;(1 + |v|) ( |\fin| + |\nabla _x \fin| + |\nabla _v \fin | ) \in \ltF.
$$
Then the solution of \eqref{Equ55} satisfies 
$$
(1 + |v|)( |f| + |\partial  _t f| + |\nabla _x f| +  |\nabla _v f| )  \in \liltF.
$$
Moreover, if $\sigma \in W^{2,\infty}$ and $(1 + |v|^2) \nabla ^2 _{x,v} \fin \in \ltF$ then $(1 + |v|^2) \nabla ^2 _{t,x,v} f \in \liltF$.
\end{pro}
\begin{proof}
By Proposition \ref{Smoothf} we know that $(1 + |v|) f \in \liltF$ and since the operators $Q_\pm$ remain bounded in $\ltF$ (see Proposition \ref{LinBolOpe})
\[
\|Q(f)\|_\ltF \leq  \frac{2S_0}{\tau} \|f\|_\ltF
\]
one gets $\|\partial _t f \|_\liltF \leq  \frac{2S_0}{\tau} \|\fin\|_\ltF$. Observe that 
\[
\intvp{s(v,\vp)f(t,x,\vp)} \leq \left ( \intvp{s(v,\vp)^2 M(\vp)}\right ) ^{1/2} \|f(t,x,\cdot)\|_\ltm \leq S_0 \|\fin (x,\cdot) \|_\ltm.
\]
Therefore we obtain
\[
\|v \partial _t f \|_\ltF = \|v Q(f) \|_\ltF \leq \frac{S_0}{\tau} \left (\intvp{|\vp|^2 M(\vp)}  \right ) ^{1/2} \;\|\fin\|_\ltF + \frac{S_0}{\tau} \|v f(t) \|_\ltF
\]
saying that $v \partial _t f \in \liltF$. Since $Q$ is commuting with the space derivatives one gets
\[
\partial _t (\nabla _x f) = Q(\nabla _x f),\;\;(t,x,v) \in \R_+ ^\star \times \R ^2 \times \R^2,\;\;\nabla _x f(0) = \nabla _x \fin \in \ker {\cal T}
\]
and therefore, by Proposition \ref{Smoothf} we obtain $(1 + |v|)|\nabla _x f | \in \liltF$. It remains to estimate the velocity derivatives. As we already know that $^\perp v \cdot \nabla _v f = 0$ let us analyze the derivative along $v$. We have
\begin{equation}
\label{EquDerVit}
\partial _t \left ( \frac{v}{|v|} \cdot \nabla _v f\right ) + \frac{1}{\tau} \left ( \frac{v}{|v|} \cdot \nabla _v f\right ) \intvp{s(v,\vp)M(\vp)} = \frac{1}{\tau} g 
\end{equation}
where
\[
g = \intvp{\frac{v}{|v|} \cdot \nabla _v s\{ M(v) f(t,x,\vp) - M(\vp) f(t,x,v)\}} + \frac{v}{|v|}\cdot \nabla _v M(v) \intvp{\!\!\!s (v,\vp) f(t,x,\vp)}.
\]
Observe that 
\[
\|g\|_\ltF \leq 2 \|\sigma ^\prime \|_{L^\infty} \|\fin \|_\ltF + S_0 \|\fin \|_\ltF  \left \| \frac{v}{|v|} \cdot \nabla _v M \right \|_\ltm
\]
and finally, solving \eqref{EquDerVit} with respect to $\frac{v}{|v|} \cdot \nabla _v f$ yields
\[
\left \| \frac{v}{|v|} \cdot \nabla _v f(t) \right \| _\ltF \leq \left \| \frac{v}{|v|} \cdot \nabla _v \fin \right \| _\ltF + \frac{1}{s_0} \left (2 \|\sigma ^{\prime}\|_{L^\infty} + S_0 \left \| \frac{v}{|v|} \cdot \nabla _v M \right \|_\ltm   \right ) \|\fin \|_\ltF.
\]
Therefore $\nabla _v f \in \liltF$. Starting from the equality $\partial _t ( v \cdot \nabla _v f) = v \cdot \nabla _v Q(f)$ we obtain the differential equation
\[
\partial _t ( v \cdot \nabla _v f) + \frac{1}{\tau} ( v \cdot \nabla _v f) \intvp{s(v,\vp) M(\vp)} = \frac{1}{\tau } h
\]
where
\[
h = \intvp{v \cdot \nabla _v s \{M(v) f(t,x,\vp) - M(\vp) f(t,x,v)\}} + v \cdot \nabla _v M \intvp{s(v,\vp) f(t,x,\vp)}.
\]
Observing that 
\[
\|h(t)\|_\ltF \leq \|\sigma ^\prime\|_{L^\infty} \|v M\|_\ltm \|\fin \|_\ltF + \|\sigma ^\prime\|_{L^\infty} \|vf(t) \|_\ltF + S_0 \|v \cdot \nabla _v M \|_\ltm \|\fin \|_\ltF
\]
we deduce that $\|v \cdot \nabla _v f(t) \|_\ltF \leq \|v \cdot \nabla _v \fin \|_\ltF + 1/s_0 \|h\|_{\liltF}$ which says that $|v|\;|\nabla _v f| = |v \cdot \nabla _v f| \in \liltF$. The last statement follows similarly.
\end{proof}
\begin{pro}
\label{SmoothfOne}
Assume that the hypotheses of Propositions \ref{TransportProp} and  \ref{SmoothfBis} hold true and that the electro-magnetic field $(E(x), B(x))$ belongs to $W^{1,\infty} (\R^2)$. Then we have
\[
f^1 \in \liltF,\;\;{\cal R}(f^1) \in \liltF 
\]
where $f^1$ is defined by \eqref{Equ59} and ${\cal R}(f^1) $ by \eqref{Equ60}.
\end{pro}
\begin{proof}
Since $\partial _t f = Q(f)$, notice that \eqref{Equ59} can be written
\[
v \cdot \nabla _x f + \frac{q}{m} E(x) \cdot \nabla _v f + {\cal T} f^1 = 0,\;\;\ave{f^1} = 0
\]
and therefore we have, cf. Remark \ref{CasPartB}
\[
f^1= - {\cal B} \cdot \nabla _{x,v} f = \frac{^\perp v}{\oc} \cdot \nabla _x f - \frac{q}{m\oc} \;^\perp E \cdot \nabla _v f.
\]
Our conclusions follow easily by the smoothness of $f$.
\end{proof}
Combining the previous arguments we have proved the strong convergence result stated in Theorem \ref{ConvStrong}.

\subsection{Second order model}
\label{SecOrdMod}
\noindent

As we have seen in the previous section, the leading order term in the expansion $\fe = f + \eps f^1 + ...$ satisfies the space homogeneous Boltzmann equation. The collision operator appearing in the first order limit model is still the linear Boltzmann operator. In other words the magnetic field does not affect the collision mechanism (the collision operator and the advection along the Laplace force $\omega _c ^\perp v \cdot \nabla _v$ are commuting). Our goal in this section is to take into account the first order correction. We are looking for a closure involving $f + \eps f^1$ accounting for the small departures from the equilibrium $f$ when the magnetic field grows like ${\cal O}(1/\eps)$. We expect that the magnetic field will affect the collision mechanism. We will see that the new collision operator entering the second order model is no more local in space. Naturally, the new collision operator should verify the usual physical properties like mass conservation, entropy dissipation. The derivation is completely formal and it is based on the average technique introduced before. The limit model being quite complicated we prefer to skip the technical details leading eventually to rigorous results. Nevertheless we expect that the arguments employed when studying the first order model will adapt to the second order one as well.

The linear Boltzmann equation \eqref{Equ6} reads
\begin{equation}
\label{Equ71} 
\partial _t \fe + a \cdot \nabla _{x,v} \fe + \frac{1}{\eps} {\cal T} \fe = Q(\fe)
\end{equation}
where $a = ( v, q/m \;E(x)), {\cal T} = \oc ^\perp v \cdot \nabla _v$, $\oc = qB(x)/m$. Plugging the expansion $\fe = f + \eps f^1 + ...$ in \eqref{Equ71} yields the constraint ${\cal T} f = 0$ and the evolution equations
\begin{equation}
\label{Equ72} 
\partial _t f + a \cdot \nabla _{x,v} f +  {\cal T} f^1 = Q(f)
\end{equation}
\begin{equation}
\label{Equ73} 
\partial _t f^1 + a \cdot \nabla _{x,v} f^1 +  {\cal T} f^2 = Q(f^1)
\end{equation}
\[
\vdots
\]
The evolution equation for the dominant term $f$ has been obtained by averaging \eqref{Equ72} along the flow of ${\cal T}$. Indeed, by Proposition \ref{Step1} we know that there is a vector field ${\cal A } = {\cal A} (x,v)$ such that 
\[
\ave{a\cdot \nxv f } = {\cal A} \cdot \nxv f,\;\;f \in \ker {\cal T}.
\]
Actually the average field associated to $a = (v, q/m \;E(x))$ is ${\cal A} = 0$ cf. Remark \ref{CasPartaA}, but we prefer to keep track of the vector field ${\cal A}$, which will clearly show that our technique will provide a second order closure independently  the field ${\cal A}$ vanishes or not. Nevertheless, in the final model the field ${\cal A}$ will be replaced by its value ${\cal A} = 0$. 
For any $f \in \ker {\cal T}$ we have $\ave{\partial _t f } = \partial _t \ave{f} = \partial _t f$ and by Corollary \ref{ColAve} we can write
\[
\ave{Q(f)} = Q(\ave{f}) = Q(f).
\]
Therefore averaging \eqref{Equ72} leads to the homogeneous Boltzmann equation
\begin{equation}
\label{Equ74} \partial _t f + {\cal A}\cdot \nxv f = Q(f).
\end{equation}
Clearly more information is available from \eqref{Equ72}. In order to exploit that we use the decomposition
\[
f^1 = g^1 + h^1,\;\;g^1 = \ave{f^1} \in \ker {\cal T},\;\;h^1 = f^1 - \ave{f^1} \in \ker \ave{\cdot}
\]
and therefore \eqref{Equ72} becomes
\begin{equation}
\label{Equ75}
\partial _t f + a \cdot \nxv f + {\cal T} h^1 = Q(f).
\end{equation}
The equations \eqref{Equ74}, \eqref{Equ75} yields 
\begin{equation}
\label{Equ76} (a - {\cal A} ) \cdot \nxv f + {\cal T} h^1 = 0
\end{equation}
and by Remark \ref{CasPartB} we deduce that $h^1 = - \tcalB \cdot \nxv f$ where the field $\tcalB$ writes (see Remark \ref{CasPartB})
\[
\tcalB = \frac{1}{\oc} \left ( - ^\perp v, \frac{q}{m}\; ^\perp E - \frac{v \cdot \nx \omega _c}{\oc} \;^\perp v \right ).
\]
We average now \eqref{Equ73}. Notice that we have
\[
\ave{\partial _t f^1} = \partial _t \ave{f^1} = \partial _t g^1,\;\;\ave{{\cal T} f^2} = 0,\;\;\ave{Q(f^1)} = Q(\ave{f^1}) = Q(g^1)
\]
\[
\ave{a\cdot \nxv f^1 } = \ave{a \cdot \nxv g^1} + \ave{a \cdot \nxv h^1} = {\cal A} \cdot \nxv g^1 - \ave{a \cdot \nxv ( \tcalB \cdot \nxv f)}
\]
and therefore we obtain the equation
\begin{equation}
\label{Equ78} \partial _t g^1 + {\cal A} \cdot \nxv g ^1 - \ave{a \cdot \nxv ( \tcalB \cdot \nxv f)} = Q(g^1).
\end{equation}
Naturally, the closure for $f + \eps f^1 = f + \eps g^1 + \eps h^1$ comes by combining \eqref{Equ74}, \eqref{Equ78}
\begin{eqnarray}
\label{Equ79} (\partial _t + {\cal A} \cdot \nxv ) ( f + \eps f^1) + \eps ( \ave{a \cdot \nxv } - {\cal A} \cdot \nxv  +  Q -  \partial _t )h^1 = Q( f + \eps f^1).
\end{eqnarray}
Making use of Remark \ref{Use} and Proposition \ref{CollCommB} we obtain
\begin{eqnarray}
( \ave{a \cdot \nxv } - {\cal A} \cdot \nxv  \!\!\!&+&\!\!\!  Q -  \partial _t )h^1 = - \Divxv ( f {\cal C}) - Q ( \tcalB \cdot \nxv f) + \partial _t (\tcalB \cdot \nxv f) \nonumber \\
\!\!\!& = & \!\!\! - \Divxv ( f {\cal C})  + \tcalB \cdot \nxv Q(f) - Q ( \tcalB \cdot \nxv f)  \nonumber \\
\!\!\!& = & \!\!\! - {\cal C} \cdot \nxv f 
- \tQo (f).\nonumber 
\end{eqnarray}
Finally \eqref{Equ79} becomes
\begin{equation}
\label{Equ80} (\partial _t + ({\cal A - \eps {\cal C}}) \cdot \nxv ) ( f + \eps f^1) = Q(f+\eps f^1) + \eps \tQo (f) - \eps ^2 {\cal C} \cdot \nxv f^1
\end{equation}
and since we are looking for a closure with respect to $f+\eps f^1$, or something equivalent to $f + \eps f^1$ up to a second order term ${\cal O}(\eps ^2)$, we need to define $\tQo$ over the space of all functions (not only on the subspace of constant functions along the flow of $b^0$, as it was done in Proposition \ref{CollCommB}). We denote by $Q^1$ this extension and assume for the moment that $Q^1$ is known (the exact expression will be determined in Proposition \ref{Adjoint} and Corollary \ref{Q1}). By doing that we observe that the right hand side of \eqref{Equ80} writes
\begin{equation}
\label{EquNewTerm}
( Q + \eps Q^1)(f+\eps f^1)  - \eps ^2 ( {\cal C} \cdot \nxv f^1 + Q^1 (f^1)).
\end{equation}
Getting rid of the second order term and taking into account that ${\cal A }= 0$, we may expect that solving
\begin{equation}
\label{EquSecOrdApp} \partial _t \tfe - \eps {\cal C} \cdot \nxv  \tfe = (Q + \eps Q^1 ) (\tfe)
\end{equation}
provides a second order approximation of $\fe$, since at least formally 
\[
\tfe - \fe = ( \tfe - (f + \eps f^1)) - (\eps ^2 f^2 + ...) = {\cal O}(\eps ^2).
\]
A delicate point is how to pick the extension of $\tQo$. Naturally, the global collision operator $Q + \eps Q^1$, and also the first order correction $Q^1$, should dissipate the entropy
\[
( Q^1 (f), f)_\ltF \leq 0,\;\;\forall \;f.
\]
In particular, for any smooth functions $g \in \ker {\cal T}$, $h \in \ker \ave{\cdot}$ and any $\delta >0$ we have
\begin{equation}
\label{Equ81} ( Q^1 (g + \delta h), g + \delta h) _\ltF \leq 0.
\end{equation}
Recall that one of the properties of $\tQo$ was that $(\tQo (g),g) _\ltF = 0$ and therefore \eqref{Equ81} becomes
\[
(\tQo (g), h) _\ltF + (g, Q^1(h))_\ltF + \delta ( Q^1(h),h) _\ltF \leq 0
\]
which is equivalent to
\begin{equation}
\label{Equ83} (\tQo (g), h) _\ltF + (g, Q^1(h))_\ltF = 0,\;\;( Q^1(h),h) _\ltF \leq 0,\;g \in \ker {\cal T},\;h \in \ker \ave{\cdot}.
\end{equation}
Recall also that $\tQo$ maps constant functions along the flow of $b^0$ to zero average functions. We will construct the extension of $\tQo$ such that $Q^1$ maps zero average functions to constant functions along the flow of $b^0$. Notice that in this case the inequality $(Q^1(h),h) _\ltF \leq 0$ holds true for any $h \in \ker \ave{\cdot}$ (actually it becomes equality, by the orthogonality between $\ker {\cal T}$ and $\ker \ave{\cdot}$) and that the first condition in \eqref{Equ83} uniquely defines $Q^1(h) \in \ker {\cal T}$ for any $h \in \ker \ave{\cdot}$. Moreover, the above construction will ensure the global mass conservation since $\tQo$ satisfies the local mass conservation and for any smooth zero average function $h$ we can write, recalling that $\tQo (F) = 0$
\[
\intxv{Q^1 (h) (x,v) } = (Q^1 (h), F)_\ltF = - (h,\tQo(F))_\ltF = 0.
\] 
\begin{pro}
\label{Adjoint} For any smooth functions $g \in \ker {\cal T}, h \in \ker \ave{\cdot}$ we have
\[
(\tQo (g),h)_\ltF = - \frac{1}{\tau} (g,K(h))_\ltF
\]
where
\begin{eqnarray}
K(h) & = & \intvp{\left ( s_1 - \frac{|v|}{|\vp|} s_2\right )M(v) \;^\perp \vp \cdot \nx \left ( \frac{h(x,\vp)}{\oc}\right )} \nonumber \\
& - & \intvp{ s_2 \frac{m}{\theta} \left [|v| M(v) \frac{\vp \cdot v_\wedge}{|\vp|} \;h(x,\vp) - |\vp| M(\vp) \frac{v \cdot v_\wedge}{|v|} \;h(x,v)\right ]}. \nonumber 
\end{eqnarray}
In particular the unique function $Q^1(h) \in \ker {\cal T}$ satisfying 
\begin{equation}
\label{Equ85} (\tQo (g), h)_\ltF + (g, Q^1 (h))_\ltF = 0,\;\;\forall \;g \in \ker {\cal T}
\end{equation}
is given by
\begin{eqnarray}
Q^1(h) & = & \frac{1}{\tau} \intvp{\left ( s_1 - \frac{|v|}{|\vp|} s_2\right )M(v) \Divx\left ( \frac{\ave{^\perp \vp  h}(x,\vp)}{\oc}\right )} \nonumber \\
& - & \frac{1}{\tau} \intvp{ s_2 \frac{m}{\theta} \left [|v| M(v) \;v_\wedge \cdot \frac{\ave{\vp h}(x,\vp)}{|\vp|}  - |\vp| M(\vp) \;v_\wedge \cdot\frac{\ave{v h}(x,v)}{|v|} \right ]}. \nonumber 
\end{eqnarray}
\end{pro}
\begin{proof}
By Proposition \ref{CollCommB} we know that
\begin{eqnarray}
\tQo (g) & = & \frac{1}{\tau} \intvp{\left ( s_1 - \frac{|\vp|}{|v|} s_2\right )M(v)F(x,\vp) \frac{^\perp v}{\oc} \cdot \nx \left ( \frac{g(x,\vp)}{F(x,\vp)}\right ) } \nonumber \\
& + & \frac{1}{\tau} \intvp{ \frac{|\vp|}{|v|}s_2 \frac{m}{\theta} (v \cdot v_\wedge ) M(v) F(x,\vp) \left [\frac{g(x,\vp)}{F(x,\vp)} -  \frac{g(x,v)}{F(x,v)}\right ]}.\nonumber
\end{eqnarray}
Integrating by parts with respect to $x$ leads to
\begin{eqnarray}
\!\!\!&  &\!\!\!\tau ( \tQo (g), h)_\ltF =  - \intxv{\intvp{\left ( s_1 - \frac{|\vp|}{|v|}s_2 \right )M(\vp) \;^\perp v \cdot \nx \left ( \frac{h(x,v)}{\oc}\right )\frac{g(x,\vp)}{F(x,\vp)}}} \nonumber \\
\!\!\!& + &\!\!\! \intxv{\intvp{ \frac{|\vp|}{|v|}s_2 \frac{m}{\theta} (v \cdot v _\wedge) \{M(v) g(x,\vp) - M(\vp) g(x,v)\}\frac{h(x,v)}{F(x,v)}}} \nonumber \\
\!\!\!& = &\!\!\! - (g,K(h))_\ltF.\nonumber 
\end{eqnarray}
Since 
$$
\tau (g, Q^1(h))_\ltF = - \tau (\tQo(g),h)_\ltF = (g,K(h))_\ltF = (g, \ave{K(h)})_\ltF 
$$
and $Q^1 (h) \in \ker {\cal T}$, we deduce 
that  $Q^1(h) = \ave{K(h)}/\tau$. For further transformations of the terms in $K(h)$ we need the easy lemma
\begin{lemma}
\label{IntAve}
Let $k = k(v,\vp)$ be a bounded function depending only on $r = |v|$ and $\rp = |\vp|$. Then for any function $f = f(x,v) \in \ltF$ we have $\intvp{k(v,\vp) M(v) f(x,\vp)} \in \ltF$, $\intvp{k(v,\vp) M(\vp) f(x,v)} \in \ltF$ and
\[
\ave{\intvp{k(v,\vp) M(v) f(x,\vp)}} = \intvp{k(v,\vp)M(v) \ave{f}(x,\vp)}
\]
\[
\ave{\intvp{k(v,\vp) M(\vp) f(x,v)}} = \intvp{k(v,\vp)M(\vp) \ave{f}(x,v)}.
\]
\end{lemma}
\begin{proof}
Observe that $\intvp{k(v,\vp)M(v) \ave{f}(x,\vp)}$, $\intvp{k(v,\vp)M(\vp) \ave{f}(x,v)}$ depend  only on $x$ and $|v|$ and therefore they belong to $\ker {\cal T}$. Pick now $\varphi = \varphi (x,v) \in \ltF$ a constant function along the flow of $b^0$. We can write, observing that $(x,\vp) \to \intv{k(v,\vp) M(v) \varphi (x,v)} \in \ker {\cal T}$
\begin{eqnarray}
\intxvbis{\varphi  \intvp{k(v,\vp) M(v) f(x,\vp)}} \!\!\!&=&\!\!\! \intxvp{f(x,\vp)\intv{k(v,\vp) M(v) \varphi(x,v)}} \nonumber \\
\!\!\!&=&\!\!\!\intxvp{\!\!\ave{f}(x,\vp)\!\!\intv{k(v,\vp) M(v) \varphi(x,v)}} \nonumber \\
\!\!\!&=&\!\!\!\intxvbis{\varphi (x,v) \intvp{k(v,\vp) M(v) \ave{f}(x,\vp)}}\nonumber
\end{eqnarray}
saying that $\ave{\intvp{k(v,\vp) M(v) f(x,\vp)}} = \intvp{k(v,\vp)M(v) \ave{f}(x,\vp)}$. The second statement follows immediately since $\intvp{k(v,\vp) M(\vp) } \in \ker {\cal T}$ and thus
\[
\ave{\intvp{k(v,\vp) M(\vp) f(x,v)}} = \intvp{k(v,\vp)M(\vp)} \ave{f}(x,v).
\]
\end{proof}
Appealing to Lemma \ref{IntAve} and Remark \ref{CasPartDiv} we obtain
\begin{eqnarray}
& & \ave{\intvp{\left ( s_1 - \frac{|v|}{|\vp|}s_2 \right )M(v)\Divx \left ( \frac{^\perp \vp h(x,\vp)}{\oc}\right ) } } \nonumber \\
& & = \intvp{\left ( s_1 - \frac{|v|}{|\vp|}s_2 \right )M(v)\Divx \ave{ \frac{^\perp \vp h(x,\vp)}{\oc}}} \nonumber 
\end{eqnarray}
and
\begin{eqnarray}
& & \ave{\intvp{s_2 \frac{m}{\theta} \left [ |v|M(v) \left ( \frac{\vp}{|\vp|} \cdot v_\wedge \right ) h(x,\vp) - |\vp|M(\vp) \left ( \frac{v}{|v|} \cdot v_\wedge \right ) h(x,v)\right] }} \nonumber \\
& & = \intvp{s_2 \frac{m}{\theta} \left [ |v|M(v) v_\wedge \cdot \ave{\frac{\vp}{|\vp|} h}(x,\vp) - |\vp|M(\vp) v_\wedge \cdot \ave{\frac{v}{|v|}  h}(x,v)\right] } \nonumber
\end{eqnarray}
implying that 
\begin{eqnarray}
Q^1(h) & = & \frac{1}{\tau} \intvp{\left ( s_1 - \frac{|v|}{|\vp|} s_2\right )M(v) \Divx\left ( \frac{\ave{^\perp \vp  h}(x,\vp)}{\oc}\right )} \nonumber \\
& - & \frac{1}{\tau} \intvp{ s_2 \frac{m}{\theta} \left [|v| M(v) \;v_\wedge \cdot \frac{\ave{\vp h}(x,\vp)}{|\vp|}  - |\vp| M(\vp) \;v_\wedge \cdot\frac{\ave{v h}(x,v)}{|v|} \right ]}. \nonumber 
\end{eqnarray}
\end{proof}
\begin{coro}
\label{Q1} The extension $Q^1$ of the operator $\tQo$ satisfying \eqref{Equ85} is given by
\begin{eqnarray}
\label{Form}
Q^1(f) & = & \frac{1}{\tau} \intvp{\left ( s_1 - \frac{|\vp|}{|v|} s_2\right )M(v) \;v \cdot \left (\frac{m v_\wedge }{\theta}\ave{f}(x,\vp) - \frac{^\perp \nabla _x \ave{f} (x,\vp)}{\oc}   \right ) }\nonumber \\
& + & \frac{1}{\tau} \intvp{s_2 \frac{|\vp|}{|v|} \frac{m}{\theta} (v \cdot v_\wedge)\{ M(v) \ave{f}(x,\vp) - M(\vp) \ave{f}(x,v)\}} \nonumber \\
& + & \frac{1}{\tau} \intvp{\left ( s_1 - \frac{|v|}{|\vp|} s_2\right )M(v) \;\Divx\left ( \frac{\ave{^\perp \vp  f}(x,\vp)}{\oc}\right )} \nonumber \\
& - & \frac{1}{\tau} \intvp{ s_2 \frac{m}{\theta} \;v_\wedge \cdot \left [|v| M(v)  \frac{\ave{\vp f}(x,\vp)}{|\vp|}  - |\vp| M(\vp) \frac{\ave{v f}(x,v)}{|v|} \right ]}. 
\end{eqnarray}
Moreover, when the scattering cross section is constant we have
\begin{equation}
\label{FormBis}
Q^1(f) = \frac{s}{\tau} M(v) \left [\Divx \left ( \frac{^\perp j_f (x)}{\oc} \right ) + \frac{\rho _F (x)}{\oc} \;^\perp v \cdot \nx \left ( \frac{\rho _f (x)}{\rho _F (x)} \right )    \right ]
\end{equation}
where $\rho _f = \intvp{f}$, $\rho _F = \intvp{F}$, $j_f = \intvp{\vp f}$.
\end{coro}
\begin{proof}
Using the decomposition $f = \ave{f} + ( f - \ave{f})$, with $\ave{f} \in \ker {\cal T}$ and $f - \ave{f} \in \ker \ave{\cdot}$ and taking into account that $\ave{v(f - \ave{f})} = \ave{vf}$ we can write
\[
Q^1(f) = \tQo (\ave{f}) + Q^1 ( f - \ave{f})
\]
leading to the formula \eqref{Form}. In particular, when the scattering cross section is constant, we have $s_1 = s$, $s_2 = 0$ and \eqref{Form} reduces to \eqref{FormBis}.
\end{proof}
\begin{remark}
\label{SecondOrderIdea} Replacing the vector field ${\cal C}$ by \eqref{EquCComponent}, the collision operator $Q$ by \eqref{EquDomCol} and the extension operator $Q^1$ by \eqref{FormBis} in the equation \eqref{EquSecOrdApp} leads to the second order model stated in Theorem \ref{ThSecOrdApp} (case of constant scattering cross section).
\end{remark}
It is easily seen, at least for constant scattering cross section, that the kernel of $Q+\eps Q^1$ reduces to $\R F = \{\lambda F:\lambda \in \R\}$. Let us establish this property in the general case.
\begin{pro}
\label{KernelQQone}
For any $\eps >0$ the only functions in the kernel of $Q + \eps Q^1$ are those proportional to $F$.
\end{pro}
\begin{proof}
Let $f$ be such that $(Q + \eps Q^1)( f) = 0$. Based on the properties of $Q$ and $Q^1$ we have
\[
(Q(f), f)_\ltF = (Q(f) + \eps Q^1 (f), f)_\ltF = 0
\]
and
\[
(Q(f), f)_\ltF = - \frac{1}{2\tau}\intxv{\intvp{s(v,\vp)M(v) F(x,\vp) \left (\frac{f(x,\vp)}{F(x,\vp)} -  \frac{f(x,v)}{F(x,v)}\right ) ^2}}
\]
implying that $f(x,v) = \rho (x) M(v)$ for some function $\rho = \rho (x)$. Therefore $Q(f) = 0$ and thus $Q^1 (f) = 0$. It is easily seen that for $f(x,v) = \rho (x) M(v)$ the formula \eqref{Form} reduces to 
\[
Q^1(f) = \frac{1}{\tau} M(v) \frac{v}{|v|} \cdot \intvp{\left ( \frac{mv_\wedge}{\theta} \rho (x) - \frac{^\perp \nabla _x \rho }{\oc} \right )(|v| s_1 - |\vp| s_2 ) M(\vp)}.
\]
We are ready if we prove that there is $r>0$ such that 
\[
J(v):= \intvp{(|v| s_1 - |\vp| s_2 ) M(\vp)}>0,\;\;|v| = r
\]
since in that case we deduce that 
\[
- \frac{\rho _F (x)}{\oc} \;^\perp \nabla _x \left (\frac{\rho (x)}{\rho _F (x)}   \right ) = \frac{mv_\wedge}{\theta} \rho (x) - \frac{^\perp \nabla _x \rho }{\oc} = 0
\]
saying that $\rho = C \rho _F$ for some constant $C$, or $f(x,v) = C\rho _F (x) M(v) = CF(x,v)$.
Recall that $s_1, s_2$ depend only on $r = |v|, \rp = |\vp|$ and
\[
s_1 = \frac{1}{2\pi} \int _0 ^{2\pi} \sigma (\sqrt{\chi})\;\mathrm{d}\alpha,\;\;s_2 = \frac{1}{2\pi} \int _0 ^{2\pi} \sigma (\sqrt{\chi})\cos \alpha \;\mathrm{d}\alpha,\;\;\chi = r^2 + (\rp)^2 - 2r\rp \cos \alpha.
\]
Therefore one gets
\[
J(r) = \int_{\R_+} \int _0 ^{2\pi} ( r - \rp \cos \alpha ) \sigma ( \sqrt{\chi}) \mu (\rp) \;\mathrm{d}\alpha\mathrm{d}\rp = \frac{1}{2}\int_{\R_+} \int _0 ^{2\pi} \frac{\partial \chi}{\partial r} \sigma ( \sqrt{\chi}) \mu (\rp) \;\mathrm{d}\alpha\mathrm{d}\rp
\]
where $\mu (\rp) = m/(2\pi \theta) \exp ( - \frac{m (\rp)^2}{2\theta} )$. Integrating with respect to $r \in \R_+$ yields
\[
\int _{\R_+} J(r) \;\mathrm{d}r = \frac{1}{2}\int_{\R_+} \int _0 ^{2\pi} \mu (\rp) \int _{(\rp)^2} ^ {+\infty} \sigma (\sqrt{\chi})\;\mathrm{d} \chi \mathrm{d}\alpha \mathrm{d}\rp >0. 
\]
Therefore there is $r^\star >0$ such that $J(r ^\star) >0$.
\end{proof}
Another point to be discussed is related to the continuity equation. Since for any $\eps >0$ the solution $\fe$ of the linear Boltzmann equation \eqref{Equ6} satisfies a equation like
\[
\partial _t \intv{\fe} + \Divx \intv{v \fe} = 0
\]
we should retrieve a similar one when integrating the second order approximation \eqref{EquSecOrdApp} with respect to $v \in \R^2$. In the case of constant scattering cross section we obtain
\[
\partial _t \intv{\tfe} + \eps \Divx \intv{(v_\wedge + v_{\mathrm{GD}})\tfe} = \eps\intv{Q^1(\tfe)} = \eps \frac{s}{\tau} \Divx  \intv{\frac{^\perp v}{\omega_c} \tfe}   
\]
and the continuity equation accounts for a new drift $v _{\mathrm{QD}} = - \frac{s}{\tau} \frac{^\perp v}{\oc}$, orthogonal to the magnetic lines, due to collisions
\[
\partial _t \intv{\tfe} + \eps \Divx \intv{(v_\wedge + v_{\mathrm{GD}} +v_{\mathrm{QD}})\tfe} = 0.
\]
In the general case, it is easily seen that 
\[
\intv{Q^1(\tfe)} = \frac{1}{\tau} \Divx \intv{\intvp{\left ( s_1 - \frac{|v|}{|\vp|} s_2 \right ) M(v) \frac{^\perp \vp}{\omega _c} \tfe (x,\vp)}} = - \Divx \intv{v_{\mathrm{QD}} \tfe}
\]
and therefore the collision drift is given by
\begin{equation}
\label{EquDriftNonConstantS}
v_{\mathrm{QD}} = - \frac{1}{\tau} \intvp{\left ( s_1 - \frac{|\vp|}{|v|} s_2 \right )M(\vp)} \;\frac{^\perp v}{\oc}.
\end{equation}
The previous arguments provide the second order approximation model, as stated by Theorem \ref{ThSecOrdApp}. For the sake of the presentation we assume here that the scattering cross section is constant.
\begin{proof} (of Theorem \ref{ThSecOrdApp})
Consider $f$ the solution of the problem $\partial _t f = Q(f), t>0$ with the initial condition $f(0) = \fin$ and $g^1$ the solution of the problem
\[
\partial _t g^1 + \ave{a\cdot \nxv h^1} = Q(g^1),\;\;t>0,\;\;g^1(0) = 0
\]
where $h^1 = - \tcalB \cdot \nxv f$. Notice that ${\cal T} g^1 = 0$, $\ave{h^1} = 0$ and by the previous computations we know that 
\[
\partial _t f + a \cdot \nxv f + {\cal T} f^1 = Q(f),\;\; f^1 = g^1 + h^1.
\]
We take also a smooth function $f^2$ such that 
\[
\partial _t f^1 + a \cdot \nxv f^1 + {\cal T} f^2 = Q(f^1)
\]
which is possible thanks to the evolution equation satisfied by $g^1$. As before we write
\[
(\partial _t + a\cdot \nxv + \eps ^{-1} {\cal T}) ( f + \eps f^1 + \eps ^2 f^2) = Q(f + \eps f^1 + \eps ^2 f^2) + \eps ^2 (\partial _t + a\cdot \nxv - Q)(f^2).
\]
Combining to the equation satisfied by $\fe$ we obtain, with the notation $\re = \fe - (f + \eps f^1 + \eps ^2 f^2)$
\[
(\partial _t + a\cdot \nxv + \eps ^{-1} {\cal T}) \re = Q(\re) + \eps ^2 (Q - \partial _t - a\cdot \nxv ) f^2.
\]
Using the entropy inequality we deduce that, for some constant $C_T$
\[
\frac{1}{2}\frac{\mathrm{d}}{\mathrm{d}t} \|\re \|_\ltF ^2 \leq C_T \eps ^2 \|\re \|_\ltF,\;\;t \in [0,T],\;\;\eps >0
\]
implying that
\begin{equation}
\label{Equ93} \|\re (t)\|_\ltF \leq \|\re (0) \|_\ltF + T C_T \eps ^2,\;\;t \in [0,T],\;\;\eps >0.
\end{equation}
Notice that 
\begin{eqnarray}
\re(0) &=& \fe _{\mathrm{in}} - (f(0) + \eps f^1 (0) + \eps ^2 f^2 (0)) = \fe _{\mathrm{in}} - (\fin + \eps h^1 (0)) + {\cal O}(\eps ^2) \nonumber \\
& = & \fe _{\mathrm{in}} - (\fin - \eps \tcalB \cdot \nxv \fin ) + {\cal O}(\eps ^2) = {\cal O}(\eps ^2) \nonumber 
\end{eqnarray}
and therefore
\begin{equation}
\label{Equ94} \sup _{0 \leq t \leq T} \|\fe (t) - ( f(t) + \eps f^1 (t))\|_\ltF = {\cal O}(\eps ^2).
\end{equation}
We also know (cf. \eqref{Equ80}, \eqref{EquNewTerm}) that $f + \eps f ^1$ satisfies
\[
(\partial _t - \eps {\cal C} \cdot \nxv ) (f + \eps f^1 ) = (Q + \eps Q^1) (f + \eps f^1) - \eps ^2 ( {\cal C} \cdot \nxv f^1 + Q^1 (f^1))
\]
which is a second order perturbation of the equation \eqref{EquSecOrdApp} (or \eqref{Equ91}) satisfied by $\tfe$, implying that 
\[
(\partial _t - \eps {\cal C} \cdot \nxv ) \tre = (Q+ \eps Q^1)(\tre) + \eps ^2 ({\cal C} \cdot \nxv f^1 + Q^1(f^1)),\;\;\tre = \tfe - (f + \eps f^1).
\]
Using now the entropy inequality of $Q + \eps Q^1$ yields
\[
\frac{1}{2}\frac{\mathrm{d}}{\mathrm{d}t} \|\tre \|_\ltF ^2 \leq \tilde{C}_T \eps ^2 \|\tre \|_\ltF,\;\;t \in [0,T],\;\;\eps >0.
\]
Taking into account that $\tre(0) = \tfe (0) - (f(0) + \eps h^1(0)) = \tfe (0) - (\fin - \eps \tcalB \cdot \nxv \fin) = 0$ we obtain
\[
\|\tre (t)\|_\ltF \leq T \tilde{C}_T \eps ^2,\;\;t \in [0,T],\;\;\eps >0
\]
saying that 
\begin{equation}
\label{Equ95}
\sup _{0 \leq t \leq T} \|\tfe (t) - ( f(t) + \eps f^1 (t))\|_\ltF = {\cal O}(\eps ^2).
\end{equation}
Our conclusion follows immediately from \eqref{Equ94}, \eqref{Equ95}.
\end{proof}

\appendix
\section{Proofs of Propositions \ref{prime-integrals}, \ref{Actions}}
\label{A}
\begin{proof} (of Proposition \ref{prime-integrals})\\
1. 
We show that $y \to (\psi _0 (y), ...,\psi _{m-1} (y))$ is a change of coordinates. Indeed, if $y, \overline{y} \in \R^m$ verify $\psi _i (y) = \psi _i (\overline{y})$, $i \in \{0,1,...,m-1\}$, then $y, \overline{y}$ belong to the same characteristic. Thus denoting  by $y_0$ the discontinuity point of $\psi _0$ on this characteristic,  there are $h, \overline{h} \in [0,T_c(y_0))$, with  $h \leq \overline{h}$ without loss of generality, such that $y = Y(h;y_0), \overline{y} = Y(\overline{h};y_0)$. Integrating $(b^0 \dny \psi _0) (Y(s;y_0)) = I(Y(s;y_0)) = I(y_0)$ between $h$ and $\overline{h}$ we obtain 
\[
0 = \psi _0 (\overline{y}) - \psi _0 (y) = \psi _0 (Y(\overline{h};y_0)) - \psi _0(Y(h;y_0)) = (\overline{h} - h ) I(y_0).
\]
Therefore $h = \overline{h}$ which implies $y = \overline{y}$. We have shown that $ y \in \R^m$ is uniquely determined by $\psi _0 (y),...,\psi _{m-1} (y)$. Indeed, $y$ belongs to the characteristic associated to the invariants $\psi _1 (y), ...,\psi _{m-1} (y)$ and, if we denote by $y_0$ the discontinuity point of $\psi _0$ on this characteristic, we have $y = Y(s;y_0)$, where the parameter $s \in [0,T_c(y))$ is determined by 
\[
\psi _0 (y) - \psi _0 (y_0) = \int _0 ^s (b^0 \dny \psi _0) (Y(\tau;y_0))\;\mathrm{d}\tau = s I(y_0).
\]
Finally, without loss of generality we suppose that $\psi _0 (y_0) = 0$ and thus $\psi _0 (y) \in [0,T_c (y_0) I(y_0)) = [0,[\psi _0]\;) = [0,S)$. Clearly the map $y \to (\psi _0(y), ..., \psi _{m-1} (y))$ is a surjection between $\R^m$ and $[0,S) \times D$, which shows 1.\\
2. Notice that $\nabla _y \psi _0 \notin \mathrm{span} \{\nabla _y \psi _1,...,\nabla _y \psi _{m-1}\}$ since $b^0 \dny \psi _0 \neq 0$ and $b^0 \dny \psi _1 = ... = b ^0 \dny \psi _{m-1} = 0$. Thus, for any $i \in \{1, ..., m-1\}$ there is a unique vector field $b ^i$ such that 
\[
\bin \psi_j = \delta ^i _j,\;\;\qquad j \in \{0,1,...,m-1\},
\]
which proves 2. 
\end{proof}

\begin{proof} (of Proposition \ref{Actions})\\
Notice that for any $y \in \R^m$ the function 
$$
s \to u(Y(s;y)) = w(\psi _0 (Y(s;y)), \psi _1 (y)...,\psi _{m-1} (y))
$$
is continuous on $\R$. In particular this holds true for any discontinuity point $y_0$ of $\psi _0$. 
For any $y = Y(s;y_0)$, $s \in (0,T_c (y_0))$ we can write
\begin{eqnarray}
(\bzn u )(y) & = & \frac{\mathrm{d}}{\mathrm{d}s} \{ u(Y(s;y_0))\} = 
\bzn \psi _0 (y) \;\partial _{\psi _0} w(\psi _0(y),\psi _1 (y_0)...,\psi _{m-1} (y_0)) \nonumber \\
&=& I(y) \partial _{\psi _0} w(\psi _0(y),...,\psi _{m-1} (y)).\nonumber 
\end{eqnarray}
It remains to analyze the differentiability around the point $y_0$. Without loss of generality we assume that $I>0$. Taking $s >0$ one gets
\begin{eqnarray}
\label{EquR} \bzn u (y_0) & = & \lim _{s \searrow 0} \frac{1}{s} [ w (\psi _0 (Y(s;y_0)), \psi _1 (y_0),...,\psi _{m-1} (y_0)) - w (\psi _0 (y_0),...,\psi _{m-1} (y_0))] \nonumber \\
& = & \lim _{s \searrow 0} \frac{1}{s} [ w (\psi _0 (y_0) + sI, \psi _1 (y_0),...,\psi _{m-1} (y_0)) - w (\psi _0 (y_0),...,\psi _{m-1} (y_0))] \nonumber \\
& = & I(y_0) \partial _{\psi _0} w _+ \;(\psi _0 (y_0),...,\psi _{m-1} (y_0))
\end{eqnarray}
where $\partial _{\psi _0} w_+$ stands for the right derivative of $w$ with respect to $\psi _0$. Taking now $s<0$, using  the $S$-periodicity of $w$ with respect to $\psi _0$, we obtain
\begin{eqnarray}
\label{EquL} \bzn u (y_0) & = & \lim _{s \nearrow 0} \frac{1}{s} [ w (\psi _0 (Y(T_c + s;y_0)),...,\psi _{m-1} (y_0)) - w (\psi _0 (y_0),...,\psi _{m-1} (y_0))] \nonumber \\
& = & \lim _{s \nearrow 0} \frac{1}{s} [ w (\psi _0 (y_0) + (T_c +s)I,...,\psi _{m-1} (y_0)) - w (\psi _0 (y_0),...,\psi _{m-1} (y_0))] \nonumber \\
& = & \lim _{s \nearrow 0} \frac{1}{s} [ w (\psi _0 (y_0) + sI,...,\psi _{m-1} (y_0)) - w (\psi _0 (y_0),...,\psi _{m-1} (y_0))] \nonumber \\
& = & I(y_0) \partial _{\psi _0} w _- \;(\psi _0 (y_0),...,\psi _{m-1} (y_0))
\end{eqnarray}
where $\partial _{\psi _0} w_-$ stands for the left derivative of $w$ with respect to $\psi _0$. Combining \eqref{EquR}, \eqref{EquL} we deduce that $w$ is differentiable with respect to $\psi _0$ and at any point $y \in \R^m$ 
\[
\bzn u = I(y) \partial _{\psi _0} w (\psi _0 (y),...,\psi _{m-1} (y)).
\]
Moreover, for any $i \in \{1,...,m-1\}$ we have
\[
\bin u = \sumjz \partial _{\psi _j} w \;\bin \psi _j = \partial _{\psi _i} w (\psi _0 (y),...,\psi _{m-1} (y)),\;\;y \in \R^m.
\]
\end{proof}

\vspace{1cm}

\noindent {\bf Acknowledgements.} This work was initiated during the visit of the first author at the University of Texas at Austin. The second author acknowledges partial support from NSF grant DMS 1109625. Support from the Institute for Computational Engineering and Sciences at the University of Texas at Austin is also gratefully acknowledged. 


\end{document}